\documentclass{amsart}

\usepackage{mathtools}
\usepackage{braket}
\usepackage{fixmath}
\usepackage{enumerate}
\usepackage{booktabs}
\usepackage{dsfont}
\usepackage{upgreek}
\usepackage[ruled]{algorithm2e}

\usepackage[bookmarks=true,hidelinks]{hyperref}

\usepackage[noabbrev, poorman, capitalise]{cleveref}

\theoremstyle{plain}
\newtheorem{theorem}{Theorem}[section]
\newtheorem{lemma}[theorem]{Lemma}
\newtheorem{corollary}[theorem]{Corollary}
\newtheorem{proposition}[theorem]{Proposition}

\theoremstyle{definition}

\newtheorem{remark}[theorem]{Remark}

\newtheorem{assumption}[theorem]{Assumption}

\crefname{assumption}{Assumption}{Assumptions}
\crefname{proposition}{Proposition}{Propositions}
\crefname{lemma}{Lemma}{Lemmas}
\crefname{figure}{Figure}{Figures}
\crefname{corollary}{Corollary}{Corollaries}

\numberwithin{equation}{section}

\newcommand{\algtol}[1]{\ensuremath{\varepsilon_{\text{#1}}}}

\newcommand{\N}{\mathbb{N}}
\newcommand{\R}{\mathbb{R}}
\newcommand{\Rplus}{\R_+}

\newcommand{\ldef}{\coloneqq}
\newcommand{\e}{\mathrm{e}}

\newcommand{\parameterControlDim}{N_c}

\newcommand{\constraintSet}{\colon}

\newcommand{\embedding}{\hookrightarrow}
\newcommand{\embeddingc}{\hookrightarrow_c}
\newcommand{\dist}[2]{\ensuremath{d}\ensuremath{_{#1}(#2)}}

\newcommand{\semigroup}[1]{\ensuremath{\e^{#1}}}
\newcommand{\Lap}{\upDelta}
\newcommand{\trace}{\mathrm{Tr}}

\newcommand{\MPRHilbert}[3]{\ensuremath{H^{1}(#1;#2)\cap L^{2}(#1;#3)}}

\DeclarePairedDelimiter{\abs}{\lvert}{\rvert}
\DeclarePairedDelimiter{\norm}{\lVert}{\lVert}
\DeclarePairedDelimiter{\pair}{\langle}{\rangle}
\DeclarePairedDelimiter{\inner}{(}{)}

\DeclareMathOperator*{\argmin}{arg\!\min}

\newcommand{\unitball}[2][0]{\ensuremath{\mathcal{B}_{#2}(#1)}}

\newcommand{\ProbSimplex}[1]{\mathbb{P}\ensuremath{_{#1}}}
\newcommand{\ProjSimplex}{\operatorname{\Pi}}

\newcommand{\Q}{\ensuremath{U}}
\newcommand{\Qad}{\ensuremath{U}_{ad}}
\newcommand{\ControlOp}{\ensuremath{B}}
\newcommand{\ControlMeasureSpace}{\ensuremath{(\omega,\varrho)}}
\newcommand{\Lcontrol}[2][I]{\ensuremath{L^{#2}(#1\times\omega)}}
\newcommand{\LcontrolSpatial}[1][2]{\ensuremath{L^{#1}(\omega)}}
\newcommand{\D}[1]{\,\mathrm{d}\ensuremath{#1}}

\renewcommand{\email}[1]{{e-mail: \tt#1}}

\begin{document}

\title[Time-optimality by distance-optimality for parabolic control systems]{Time-optimality by distance-optimality for parabolic\\ control systems}
\thanks{The first author acknowledges support from
	the International Research Training Group IGDK1754, funded by the DFG and FWF.
	The second author acknowledges  support by the ERC advanced grant 668998 (OCLOC) under the EU's H2020
	research program.}

\author[Lucas Bonifacius and Karl Kunisch]{Lucas Bonifacius and Karl Kunisch}

\begin{abstract}
The equivalence of time-optimal
and distance-optimal control problems is shown
for a class of parabolic control systems.
Based on this equivalence, 
an approach for the efficient algorithmic solution
of time-optimal control problems is investigated.
Numerical examples are provided to illustrate
that the approach works well is practice.	
\end{abstract}
%
%
%

\thanks{Fakult\"at f\"ur Mathematik, Technische Universit\"at M\"unchen; \email{lucas.bonifacius@tum.de}}
\thanks{Institute for Mathematics and Scientific Computing, University of Graz; \email{karl.kunisch@uni-graz.at}}

\maketitle
\section{Introduction}
This article is devoted to time optimal control problems for parabolic systems.
Specifically, we propose a formulation which is equivalent 
to the original time optimal control formulation and
amenable for numerical realization. We  consider the problem 
\begin{equation}\label{P}\tag{\ensuremath{P}}
\mbox{Minimize~} T
\quad\text{subject to}\quad
\left\{
\begin{aligned}
T &> 0\mbox{,}\\
u &\in \Qad(0,T)\mbox{,}\\
\partial_t y + A y &= Bu\mbox{,} &\mbox{in~} (0,T)\mbox{,}\\
y(0) &= y_0\mbox{,}\\
\norm{y(T)-y_d}_H &\leq \delta_0\mbox{,}\\
\end{aligned}
\right.
\end{equation}	
where $y$ denotes the
state, $u$ the control, and $T$ the terminal time.
Here, the set of admissible controls is
\[
\Qad(0,T) \ldef \left\{u \in L^2((0,T); L^2(\omega)) \constraintSet  u_a \leq u(t) \leq u_b\mbox{~a.e.~} t\in (0,T)\right\}
\]
for $u_a, u_b \in L^\infty(\omega)$ the control constraints,
where $\omega$ is a measurable set.
Moreover, $A$ is an unbounded operator satisfying
G\aa rding's inequality and 
$B$ is the (bounded) control operator; see also \cref{sec:notation_assumptions}
for the precise assumptions.
The goal is to steer the system into a ball centered at $y_d$ with radius $\delta_0$
in the shortest time possible.
Note that the state equation is posed on a variable time horizon
which causes a nonlinear dependency of the state $y$
with respect to the terminal time $T$ and the control $u$.
For this reason,~\eqref{P} is a nonlinear and nonconvex
optimization problem subject to control as well
as state constraints.
Additionally we emphasize that
the objective functional does not contain control costs
which complicates the algorithmic solution of~\eqref{P}
compared to the situation with an $L^2$ term in the objective;
cf., e.g., \cite{Ito2010,Kunisch2015}.

In this article, we propose an
equivalent reformulation in
terms of minimal distance problems
that can be used algorithmically
to solve the time-optimal control problem.
For $\delta > 0$ consider the
perturbed time-optimal control problem defined as
\begin{equation}\label{TOPT_state_perturbed}\tag{$P_\delta$}
\inf_{\substack{T > 0\\ u \in \Qad(0,T)}} \; T \quad\text{subject to}\quad \norm{y[u](T) - y_d}_H \leq \delta_0 + \delta,
\end{equation}	
where $y[u](T)$ denotes the state associated with the control $u$
evaluated at $T$.
Moreover, we consider the \emph{minimum distance control problem}
\begin{equation}\label{DOPT_intro}\tag{$\delta_T$}
\mbox{Minimize~} \norm{y[u](T) - y_d}_H - \delta_0 \quad\mbox{subject to}\quad u \in \Qad(0,T)\mbox{.}\\
\end{equation}
Under weak assumptions we show that
the associated value functions
defined by
\[
T(\delta) = \inf\eqref{TOPT}\quad\text{and}\quad
\delta(T) = \inf\eqref{DOPT}
\]
are inverse to each other; see \cref{prop:distanceoptimal_delta_continuous}.
Furthermore, we prove that~\eqref{TOPT_state_perturbed} 
and~\eqref{DOPT_intro} are equivalent.
Precisely, if $\delta > 0$ is given and $(T,\bar{u})$ is optimal for~\eqref{TOPT_state_perturbed},
then $\bar{u}$ is also distance-optimal for~\eqref{DOPT_intro}.
Conversely, if $T > 0$ is given and $\bar{u}$ is distance optimal,
then $(T,\bar{u})$ is time-optimal with $\delta = \delta(T)$; 
see \cref{thm:equivalence_topt_dopt} for details.
Hence, instead of solving the time-optimal control problem directly, 
we can search for a root of the $\delta(\cdot)$-value function.
A similar equivalence first appeared in \cite{Wang2013a}
(see also \cite[Section~5.4]{Wang2018})
for the situation where one aims at delaying 
the activation of the control as long as possible. 
However, to the best of our knowledge it has never 
been considered for an algorithmic approach.
In this regard, we also mention a similar approach used in \cite{Gugat2000}
for time-optimal control of a one-dimensional vibrating system
with controls in a subspace of $L^2$
determined by certain moment equations.	

We show that the $\delta(\cdot)$-value function
is continuously differentiable for many important control scenarios;
see \cref{sec:minimal_distance_value_function}.
If in addition qualified optimality conditions
hold for the original problem,
then the derivative of $\delta(\cdot)$
is nonvanishing near the optimal solution;
see \cref{prop:equiv_qualified_optcond_deltaprime_nonzero}.
This justifies to use a Newton method 
for the calculation of a root of $\delta(\cdot)$.
Moreover, under an additional assumption we show
that the derivative of the value function
is Lipschitz continuous which guarantees fast local convergence
of the Newton method. 
In fact, in all our numerical examples, 
we observe quadratic order of convergence,
even if the additional assumption does not hold.
For the solution of the resulting minimal distance
problem with simple control constraints, 
the literature offers a wide spectrum of algorithms.

Time optimal control problems are among the most studied problems of optimal control,
and thus it comes at no surprise that diverse techniques have been proposed for their solution.
In the following let us briefly describe some of them. 
An approach, which is conceptually close,
rests on an equivalent reformulation
utilizing minimum norm problems.
In contrast, to the perturbations in the terminal constraint as in~\eqref{TOPT},
perturbations in the control constraint are introduced. 
To explain the approach, 
we consider the time optimal control problem 
\begin{equation}\label{TOPT_control_perturbed}\tag{$P_\rho$}
\inf_{\substack{T > 0\\ u \in L^2((0,T)\times\omega)}} \; T \quad\text{subject to}\quad \norm{y[u](T) - y_d}_H \leq \delta_0,\quad \norm{u}_{L^\infty((0,T); L^2(\omega))} \leq \rho,
\end{equation}
which is related to the \emph{minimal norm problem} defined as
\begin{equation}\label{TOPT_minimal_norm}\tag{$N_T$}
\inf_{u \in L^2((0,T)\times\omega)} \;  \norm{u}_{L^\infty((0,T); L^2(\omega))} \quad\text{subject to}\quad \norm{y[u](T) - y_d}_H \leq \delta_0.
\end{equation}
To follow much of the literature,
we adapted the control constraints to be chosen in $L^\infty((0,T); L^2(\omega))$
rather than $L^\infty((0,T)\times\omega)$.
Under appropriate controllability assumptions these 
two problems have been shown to be equivalent; see \cite{Krabs1982,Fattorini2005,Gozzi1999,Wang2012a,Qin2018}
for parabolic equations,
\cite{Zhang2013} for time-varying ordinary differential equations,
\cite[Chapter~5]{Wang2018} for abstract evolution equations,
and \cite{Zhang2015} for the Schr\"odinger equation.
Note that typically these publications consider the case of exact controllability, 
i.e.\ $\delta_0 = 0$, or exact null controllability, 
i.e.\ $\delta_0 = 0$ and $y_d = 0$. 
The solution to~\eqref{TOPT_minimal_norm} can be determined
by solving an unconstrained optimization problem given by
\begin{equation}\label{TOPT_minimal_norm_variational}
\inf_{\varphi_T \in H} \frac{1}{2}\left(\int_{0}^{T}\norm{B^*\varphi(t)}_{L^2(\omega)}\D{t}\right)^2 + \inner{\varphi(0),y_0} + \delta_0 \norm{\varphi_T}_{H} - \inner{y_d,\varphi_T},
\end{equation}
where $\varphi$ is the solution to the adjoint state equation
\[
-\partial_t\varphi + A^*\varphi = 0,\quad \varphi(T) = \varphi_T;
\]
see \cite[Section~4]{Qin2018}, compare also \cite[Section~1.7]{Glowinski2008}.
If $\bar{\varphi}_T$ is the minimizer of~\eqref{TOPT_minimal_norm_variational},
then the minimum norm control is given by 
\[
\bar{u}(t) = \left(\int_{0}^{T} \norm{B^*\bar{\varphi}(t)}_{L^2(\omega)}\D{t}\right)
\frac{(B^*\bar{\varphi})(t)}{\norm{(B^*\bar{\varphi})(t)}_{L^2(\omega)}},\quad\text{a.e.~} t\in (0,T),
\]
where $\bar{\varphi}$ is the adjoint state with terminal value $\bar{\varphi}_T$.
Turning to the numerical realization,
as far as we know,
the only algorithmic studies based on this equivalence
are \cite{Lu2017} for time-optimal control problems
subject to ordinary differential equations
and \cite{Muench2013} for problems
subject to partial differential equations employing
an optimal design approach.	
A direct numerical realization of \eqref{TOPT_minimal_norm}
is impeded by the difficulties related to the appearance of the state constraint
and the fact that the minimization is carried out over a non-reflexive Banach space.
In contrast~\eqref{TOPT_minimal_norm_variational} does not contain state constraints.
Turning to the realization of~\eqref{TOPT_minimal_norm}
by means of~\eqref{TOPT_minimal_norm_variational},
one has to cope with the non-smoothness of the $L^1$-norm,
whereas~\eqref{DOPT_intro}
involves the minimization of a Hilbert space norm.	
In addition,~\eqref{TOPT_minimal_norm_variational}
can be considered as an inverse source problem
for the initial condition of the adjoint problem.
Such problems are inherently ill-posed.
For the specific context of~\eqref{TOPT_minimal_norm_variational} 
this was analyzed in \cite{Muench2010}.

An alternative approach to solve time-optimal control
problems for finite or infinite dimensional systems
is based on solving the optimality system for~\eqref{P} 
after adding a regularization term of the form $\alpha\norm{u}^2$
to the cost functional. In an additional outer loop
the regularization parameter $\alpha$ can be driven to zero;
see \cite{Ito2010,Kunisch2013a,Kunisch2015}.
This is a flexible method, but one has to cope
with the difficulties of the asymptotic behavior as
the regularization parameters tends to zero.
We compare our approach with the regularization approach
in one numerical example
and observe that even for a fixed regularization parameter
our algorithm performs roughly five to ten times faster
in terms of the required number of solves for the partial differential equation; see \cref{table:example2_costs}.

Yet another approach, which has mostly been investigated
for time-optimal control problems subject to
ordinary differential equations, rests on the reformulation of~\eqref{P}
as an optimization problem with respect to the switching points
of the optimal controls; see, e.g., \cite{Kaya2003,Meier1990}.
This approach cannot be extended to the distributed control setting
in a straightforward way.

This paper is organized as follows: In \cref{sec:notation_assumptions}
we introduce the notation and main assumptions.
The equivalence of time and distance optimal controls
is proved in \cref{sec:equiv_time_distance_optimal_controls}.
\Cref{sec:time_optimal_control} is devoted to
general properties of the time-optimal control
problem.
Differentiability of the value function
associated to the minimal distance problems
is proved in \cref{sec:minimal_distance_value_function}.
The algorithm is presented in \cref{sec:algorithm}.
Various numerical examples in \cref{sec:numerical_examples}
show that our approach is efficient in practice.
Last, in \cref{sec:open_problems} we conclude with some open problems.

\section{Notation and main assumptions}
\label{sec:notation_assumptions}

Let $V$ and $H$ be real Hilbert spaces 
forming a Gelfand triple, i.e.\ $V \embeddingc H \cong H^* \embedding V^*$,
where $\embedding$ denotes the continuous embedding and
$\embeddingc$ the continuous and compact embedding.
We abbreviate the duality pairing between $V$ and $V^*$
as well as the inner product and norm in \(H\) by
\[
\pair{\cdot,\cdot} = \pair{\cdot,\cdot}_{V^*,V}, \quad
\inner{\cdot,\cdot} = \inner{\cdot,\cdot}_H, \quad
\norm{\cdot} = \norm{\cdot}_H\mbox{.}
\]
\begin{assumption}
	Let \(a \colon V \times V \to \R\) be a continuous bilinear form, which satisfies
	the G\aa rding inequality (also referred to as weak coercivity):
	There are constants $\alpha_0 > 0$ and $\omega_0 \geq 0$ such that
	\begin{equation}
	\label{eq:A_Garding}
	a(v,v) + \omega_0\norm{v}^2 \geq \alpha_0\norm{v}^2_V\quad \text{for all~} v \in V\mbox{.}
	\end{equation}
\end{assumption}
We denote by \(A \colon V \subset V^* \to V^*\) the unique linear operator with
\[
\pair{Ay,v} = a(y,v) \quad\text{for all~} v \in V.
\]
The G\aa rding inequality implies that
$-A$ generates an analytic semigroup on $V^*$ 
denoted $\semigroup{-\cdot A}$; see, e.g., \cite[Section~1.4]{Ouhabaz2005}.
\begin{assumption}\label{assumptionControl}		
	Let $\ControlMeasureSpace$ be a measure space. 
	We assume that the control operator $B \colon L^2\ControlMeasureSpace \rightarrow V^*$ is linear and continuous.
	Moreover, $y_d \in H$ is the desired state and $\delta_0 > 0$.
\end{assumption}
The abstract measure space allows for one consistent notation
for different control scenarios. 
For example, in case of a distributed control 
on a subset $\omega$ of the spatial domain $\Omega \subset \R^d$
we take $\omega$ equipped with the Lebesgue measure.
If no ambiguity arises, we drop the measure $\varrho$ and simply write $\omega$ in the following.
The space of admissible controls is defined as
\begin{equation*}
\Qad \ldef \left\{u \in \LcontrolSpatial \constraintSet  u_a \leq u \leq u_b \;\mbox{~a.e.\ in~}\; \omega\right\} \subset \LcontrolSpatial[\infty]
\end{equation*}
for $u_a, u_b \in L^\infty(\omega)$ with $u_a < u_b$ almost everywhere.
In addition, for $T > 0$ we set $\Q(0,T) \ldef L^2((0,T)\times\omega)$ and
\begin{equation*}
\Qad(0,T) \ldef \left\{u \in \Q(0,T) \constraintSet  u(t) \in \Qad\mbox{~a.e.~} t\in (0,T)\right\} \subset \Lcontrol[(0,T)]{\infty},
\end{equation*}	
where $(0,T)\times\omega$ is equipped with the completion of 
the product measure.
For $T > 0$ we use $W(0,T)$ to abbreviate
$\MPRHilbert{(0,T)}{V^*}{V}$, endowed with the canonical norm and inner product. The
symbol $i_T \colon W(0,T) \rightarrow H$ denotes the trace mapping $i_T y = y(T)$.
For any two Banach spaces $X$ and $Y$, let $\mathcal{L}(X,Y)$ denote 
the space of linear and bounded operators from $X$ to $Y$.
The symbol $\unitball[x]{r}$
stands for the ball centered at $x \in X$ with radius $r >0$ in $X$. 
Moreover, $\Rplus$ abbreviates the open interval $(0,+\infty)$.

Last, to ensure the existence of optimal controls we require the following
\begin{assumption}\label{assumption:existence_feasible_control}
	There exist a finite time $T > 0$ and a feasible control $u \in \Qad(0,T)$ such that the solution to the state equation of \eqref{P} satisfies $\norm{y(T) - y_d} \leq \delta_0$. To exclude the trivial case, 
	we in addition assume $\norm{y_0 - y_d} > \delta_0$.
\end{assumption}

\section{Equivalence of time and distance optimal controls}
\label{sec:equiv_time_distance_optimal_controls}
Instead of solving the time-optimal control problem directly,
we propose to solve
an equivalent reformulation in terms of
minimal distance control problems.
The reformulation leads to 
a bilevel optimization problem, 
where we search for a root of a certain value function in the outer loop
and solve convex optimization problems in the inner loop.
We start by proving the equivalence of minimal time and minimal distance controls.

For any $\delta \geq 0$ we consider the perturbed time-optimal control problem
\begin{equation}\label{TOPT}\tag{$P_{\delta}$}
\begin{aligned}
\mbox{Minimize~} T \quad\mbox{subject to}\quad T&\in\Rplus\mbox{,~} 
u \in \Qad(0,T)\mbox{,}\\
\norm{y[u](T) - y_d} &\leq \delta_0 + \delta\mbox{.}
\end{aligned}
\end{equation}
Moreover, for fixed $T > 0$ we consider the \emph{minimal distance control problem}
\begin{equation}\label{DOPT}\tag{$\delta_T$}
\mbox{Minimize~} \norm{y[u](T) - y_d} - \delta_0 \quad\mbox{subject to}\quad u \in \Qad(0,T)\mbox{.}\\
\end{equation}
Note that \eqref{TOPT} is a nonlinear and nonconvex optimization problem subject 
to control as well as state constraints, whereas~\eqref{DOPT} is a convex problem subject 
to control bounds only.

We define the value functions $T \colon [0,\infty) \to [0,\infty]$ and $\delta \colon [0,\infty) \to [0,\infty)$ as
\[
T(\delta) = \inf\eqref{TOPT}\quad\text{and}\quad
\delta(T) = \inf\eqref{DOPT}.
\]
Let us formulate the main result of this section.
\begin{theorem}\label{thm:equivalence_topt_dopt}
	Let $T(\cdot)$ be left-continuous.
	If $T \in (0, T(0)]$ and $u \in \Qad(0,T)$ is distance-optimal for~\eqref{DOPT}, then $(T,u)$ is also time-optimal for $(P_{\delta(T)})$.
	Conversely, if $\delta \in [0,\delta^\bullet)$ and $(T,u) \in \Rplus\times\Qad(0,T)$ is time-optimal for~\eqref{TOPT}, then $u$ is also distance-optimal for $(\delta_T)$,
	where $\delta^\bullet = \norm{y_0 - y_d} - \delta_0$.
\end{theorem}
The proof of \cref{thm:equivalence_topt_dopt} will be given in the following.
We first note that
due to boundedness of $\Qad$, linearity of the control-to-state mapping (for fixed $T > 0$), and weak lower semicontinuity of the norm, the problem~\eqref{DOPT} is well-posed,
and for this reason the value function $\delta(\cdot)$ is well-defined. 
In contrast, to verify well-posedness of~\eqref{TOPT} we require \cref{assumption:existence_feasible_control}; cf.\ also \cref{prop:existence}.
\begin{proposition}\label{prop:T_value_function_finite}
	The value function $T$ is finite, i.e.\
	$T(\cdot) < \infty$ on $[0,\infty)$.
\end{proposition}
\begin{proof}
	Let $(T,u) \in \Rplus\times\Qad(0,T)$ be the feasible point from \cref{assumption:existence_feasible_control},
	i.e.\ $\norm{y[u](T) - y_d} \leq \delta_0$.
	Clearly, $(T,u) \in \Rplus\times\Qad(0,T)$ is also feasible for $(P_{\delta})$ for any $\delta > 0$.
	Thus, $T(\delta) \leq T < \infty$.
\end{proof}
\begin{proposition}
	\label{prop:T_right_continuous}
	Set $\delta^\bullet = \norm{y_0 - y_d} - \delta_0$. The function $T \colon [0,\delta^\bullet] \to [0,\infty)$ is strictly monotonically decreasing and right-continuous.
\end{proposition}	
\begin{proof}
	\textit{Step 1: $T$ is strictly decreasing.} Clearly, $T$ is monotonically decreasing. To show strict monotonicity, let $\delta_1 > \delta_2 \geq 0$. We have to show $T(\delta_1) < T(\delta_2)$. Suppose $T(\delta_1) = T(\delta_2)$ and let $(T(\delta_i),u_i) \in \Rplus\times\Qad(0,T(\delta_i))$ be
	optimal solutions to $(P_{\delta_i})$, $i = 1,2$. Since
	\[
	\norm{y[u_2](T(\delta_2)) - y_d} - \delta_0 = \delta_2 < \delta_1,
	\]
	we infer that $(T(\delta_2),u_2)$ is also feasible for $(P_{\delta_1})$. 
	Note that in the problem formulation we can equivalently use $\norm{y[u](T)-y_d} \leq \delta$ and $\norm{y[u](T)-y_d} = \delta$.
	From continuity of $y[u_2] \colon [0,T(\delta_2)] \to H$ and $T(\delta_1) = T(\delta_2)$
	we deduce that $(T(\delta_1),u_1)$ cannot be optimal for the time-optimal problem $(P_{\delta_1})$.
	This contradicts the assumption and we conclude $T(\delta_1) < T(\delta_2)$.
	
	\textit{Step 2: $T$ is right-continuous.} Consider a sequence $\delta_1 \geq \delta_2 \geq \ldots \geq \delta_n \to \delta$. We have to show $\lim_{n\to\infty}T(\delta_n) = T(\delta)$.
	Assume that $\lim_{n\to\infty}T(\delta_n) \neq T(\delta)$. Then, due to monotonicity of $T$, there is $\varepsilon > 0$ such that
	\[
	\lim_{n\to\infty}T(\delta_n) = T(\delta) - \varepsilon.
	\]
	Let $u_n = u_n(\delta_n,T(\delta_n)) \in \Qad(0,T(\delta_n))$ denote an optimal control to $(P_{\delta_n})$.
	We can extend each $u_n$ to the time-interval $(0,T(\delta))$ so that $u_n \in \Qad(0,T(\delta))$ for all $n \in \N$. Due to boundedness of $\Qad(0, T(\delta))$, there is a subsequence denoted in the same way such that $u_n \rightharpoonup u$ in $L^s((0,T(\delta))\times\omega)$ with $u \in \Qad(0,T(\delta))$ and some $s > 2$. 
	Now, continuity of $y[u] \colon [0,T(\delta)] \to H$ and
	the triangle inequality imply
	\begin{align*}
	\lim_{n \to \infty} \norm{y[u_n](T(\delta_n)) - y_d}
	&\geq \lim_{n \to \infty} \norm{y[u](T(\delta_n))-y_d}
	- \lim_{n \to \infty}\norm{y[u](T(\delta_n))- y[u_n](T(\delta_n))}\\
	&\geq \lim_{n \to \infty} \norm{y[u](T(\delta_n))-y_d} - \lim_{n \to \infty}\sup_{t \in [0,T(\delta)]}\norm{y[u](t)- y[u_n](t)}\\
	&= \lim_{n \to \infty} \norm{y[u](T(\delta_n))-y_d},
	\end{align*}
	where in the last step we have used compactness of the control-to-state mapping 
	from $\R\times L^s((0,T(\delta))\times\omega)$ to $C([0,T(\delta)]; H)$;
	see \cite[Proposition~A.19]{Bonifacius2018}.
	Therefore, 
	\[
	\delta + \delta_0 = \lim_{n \to \infty} \delta_n  + \delta_0 
	= \lim_{n \to \infty} \norm{y[u_n](T(\delta_n))-y_d}
	\geq \norm{y[u](T(\delta)-\varepsilon)-y_d}.
	\]
	Thus, $(T(\delta)-\varepsilon, u)$ is admissible for $(P_\delta)$, contradicting optimality of $T(\delta)$. 
\end{proof}

\begin{proposition}\label{prop:distanceoptimal_delta_continuous}
	Let $T(\cdot)$ be left-continuous. Then $\delta \colon [0,T(0)] \to [0,\infty)$ is continuous and strictly monotonically decreasing.
	Moreover,
	\begin{equation}\label{eq:topt_dopt_equiv_1}
	T(\delta(T')) = T'\quad\text{for all~} T' \in [0,T(0)]
	\end{equation}
	and
	\begin{equation}\label{eq:topt_dopt_equiv_2}
	\delta(T(\delta')) = \delta'\quad\text{for all~} \delta' \in [0,\delta^\bullet].
	\end{equation}
\end{proposition}
\begin{proof}
	First, since $T$ is strictly decreasing, its inverse $T^{-1}$ is continuous. 
	Moreover, as $T$ is right-continuous according to \cref{prop:T_right_continuous},
	the assumption implies that $T$ is continuous.
	Hence, $T^{-1}$ is defined everywhere on $[0, T(0)]$; see, e.g., \cite[Theorem~III.5.7]{Amann2005a}.
	
	Let $T > 0$. Then there exists $u \in \Qad(0,T)$ such that $\norm{y[u](T)-y_d} - \delta_0 = \delta(T)$.
	Hence, $T(\delta(T)) \leq T$ holds. 
	Suppose that $T(\delta(T)) < T$. Then by continuity of $T$ there exists $\delta' < \delta(T)$ such that $T(\delta') = T$. Let $u' \in \Qad(0,T)$ be an optimal control to $(P_{\delta'})$. Then
	\[
	\delta' < \delta(T) \leq \norm{y[u'](T)-y_d}  - \delta_0 \leq \delta',
	\]
	a contradiction, which proves~\eqref{eq:topt_dopt_equiv_1}.
	
	Moreover, \eqref{eq:topt_dopt_equiv_1} implies that $T(\delta(T(\delta'))) = T(\delta')$ for all $\delta' \in [0,\delta^\bullet]$. Strict monotonicity of $T$ therefore yields \eqref{eq:topt_dopt_equiv_2}.
	For these reasons, $\delta = T^{-1}$ and we conclude that $\delta$ is continuous and strictly monotonically decreasing.
\end{proof}
After this preparation we can now prove the equivalence of time and distance optimal controls. 
\begin{proof}[Proof of \cref{thm:equivalence_topt_dopt}]
	Let $T > 0$ and $u \in \Qad(0,T)$ be distance-optimal for~\eqref{DOPT}, i.e.\ $\delta(T) = \norm{y[u](T)-y_d} - \delta_0$.
	Due to \eqref{eq:topt_dopt_equiv_1} we have $T(\delta(T)) = T$. Thus, $(T,u)$ is also time-optimal for $(P_{\delta(T)})$.
	
	Conversely, let $\delta \geq 0$ and $(T,u) \in \Rplus\times\Qad(0,T)$ be time-optimal for~\eqref{TOPT}. In particular, 
	this gives $\norm{y[u](T)-y_d} - \delta_0 = \delta$. Using \eqref{eq:topt_dopt_equiv_2} we infer that
	\[
	\delta(T(\delta)) = \delta = \norm{y[u](T)- y_d} - \delta_0,
	\]
	i.e.\ $u$ is also distance-optimal for $(\delta_T)$.
\end{proof}

Since monotone functions have at most countably many discontinuities, 
see, e.g.,  \cite[Proposition~III.5.6]{Amann2005a}, 
it is unlikely that we accidentally hit a point where $T$ is not left-continuous.
However, for the algorithm to be presented later we are interested in 
continuity of the value function
in a neighborhood of the optimal value.
To this end, we state two sufficient conditions.
Note that the second condition even guarantees 
Lipschitz continuity from the left of the value function $T(\cdot)$. 
The setting considered in \cite{Wang2012a} for example
automatically satisfies the assumptions of \cref{prop:T_Lipschitzfromleft_sufficient},
since $y_d = 0$ and $0 \in \Qad$.
\begin{proposition}
	\label{prop:T_continuousfromleft_sufficient}
	Let the following controllability condition hold:
	For all $\delta \in (0,\delta^\bullet]$
	there exists $T' > 0$ 
	and a control $u \in \Qad(0,T')$
	such that the trajectory $y[u, y^\delta(T(\delta))]$ with initial value $y^\delta(T(\delta))$ satisfies
	\[
	\norm{y[u, y^\delta(T(\delta))](t) - y_d} < \delta_0 + \delta \quad\text{for all~} t \in (0,T'],
	\]
	where $y^\delta \in W(0,T(\delta))$ 
	denotes an optimal trajectory for \eqref{TOPT}.
	Then, $T$ is left-continuous.
\end{proposition}
\begin{proof}
	Let $\delta > 0$ and
	let $u \in \Qad(0, T(\delta))$ be an optimal control to~\eqref{TOPT}.
	According to the controllability assumption
	there exists $T' > 0$ and an extended control $u' \in \Qad(0, T(\delta) + T')$,
	i.e.\ $u' = u$ on $(0,T(\delta))$,
	such that
	\[
	\norm{y[u', y_0](T(\delta)+t) - y_d} < \delta_0 + \delta
	\quad\text{for all~}t \in (0,T'].
	\]
	Set $f(t) \ldef \norm{y[u'](T(\delta)+t) - y_d} - \delta_0$.
	Since $f$ is continuous, $f(0) = \delta$, and $f(t) < \delta$ for $t \in (0,T']$, 
	for all $\delta_n > 0$ sufficiently small 
	there exists $t_n \in (0,T']$ such that
	$f(t_n) = \delta - \delta_n$.
	Hence, $(T(\delta) + t_n, u')$ is feasible for $(P_{\delta-\delta_n})$
	and we have
	\[
	T(\delta - \delta_n) \leq T(\delta) + t_n.
	\]
	Moreover, if $\delta_n \to 0$ 
	then $t_n \to 0$ due to $f(t) < \delta$ for $t \in (0,T']$.
	Passing to the limit $\delta_n \to 0$
	yields the result.		
\end{proof}

\begin{proposition}
	\label{prop:T_Lipschitzfromleft_sufficient}
	Let $y_d \in V$ and assume that G\aa rding's inequality~\eqref{eq:A_Garding} 
	holds with \(\omega_0 = 0\).
	If there exists a control $\breve{u}\in \Qad$
	such that $\norm{B\breve{u}-Ay_d}_{V^*} < \alpha_0\delta_0$,
	then $T \colon [0,\delta^\bullet] \to \R$ is Lipschitz continuous from the left.
\end{proposition}
\begin{proof}
	We argue similarly as in \cite[Theorem~4.5]{Bonifacius2017}.
	First, the assumptions of \cref{prop:T_Lipschitzfromleft_sufficient} ensure that \cite[(3.3)]{Bonifacius2017}
	holds with $h_0 = \alpha_0\delta_0 - \norm{B\breve{u}-Ay_d}_{V^*} > 0$;
	see the proof of \cite[Proposition~5.3]{Bonifacius2017}.
	
	Let $\delta \in [0,\delta^\bullet]$, $\delta' \in [0, \delta]$, and
	$(T,u,y)$ be a solution to problem $(P_{\delta})$.
	Consider the auxiliary problem \(\partial_t \breve{y} + A
	\breve{y} = B \breve{u}\) with initial condition \(\breve{y}(0) = y(T)\)
	and an auxiliary control \(\breve{u}\colon [0,\infty) \to \Qad\). Employing \cite[Lemma~3.9]{Bonifacius2017}
	we can choose \(\breve{u}\) such that
	\[
	\max\set{0, \norm{\breve{y}(t) - y_d} - \delta_0} \leq \max\set{0,\;\delta - h_0 t}
	\quad\text{for~} t\geq 0
	\]
	holds, 
	since
	\(\dist{U}{y} = \max\set{0, \norm{y - y_d} - \delta_0}\),
	where $\dist{U}{\cdot}$ denotes the distance function to the set $U = \unitball[\delta_0]{y_d}$.
	Choose
	\(\tau = (\delta - \delta')/h_0\).
	Then, \(u \in \Qad(0,T+\tau)\) defined by
	\begin{align*}
	u(t) &= \begin{cases}
	u(t) &\text{if~} t \leq T\mbox{,}\\
	\breve{u}(t-T) &\text{if~} t > T\mbox{,}
	\end{cases}
	\end{align*}
	is admissible for $(P_{\delta'})$ and we find
	\begin{equation*}
	T(\delta') \leq T(\delta) + \tau 
	= T(\delta) + (\delta - \delta')/h_0
	\end{equation*}
	concluding the proof.
\end{proof}

Instead of solving the 
time-optimal control problem, 
we can equivalently search for a root of the value function $\delta(\cdot)$
by virtue of \cref{thm:equivalence_topt_dopt}.
However, this still might be a difficult task, as $\delta(\cdot)$ can have 
several roots and we do not know the approximate region of the root we are looking for. 
If in addition the target set $\unitball[y_d]{\delta_0}$ is weakly invariant under $(A,B\Qad)$,
i.e.\ for every $y_0 \in \unitball[y_d]{\delta_0}$ there exists a control $u \colon
[0,\infty) \rightarrow \Qad$ such that the solution $y$ to
\begin{equation*}
\partial_t y + Ay = Bu\mbox{,}\quad y(0) = y_0\mbox{,}
\end{equation*}
satisfies $y(t) \in \unitball[y_d]{\delta_0}$ for all $t \geq 0$,
then there is only one root where $\delta(\cdot)$ changes from
a strictly positive value to a nonpositive value.
We also refer to \cite[Theorem~3.8]{Bonifacius2017}
for a characterization of weak invariance.
\begin{proposition}
	Let $T \in \Rplus$ such that $\delta(T) = 0$.
	Suppose that the target set $\unitball[y_d]{\delta_0}$ is weakly invariant under $(A,B\Qad)$.		
	Then, $\delta(T + t) \leq 0$ for all $t \in \Rplus$.
\end{proposition}
\begin{proof}
	This immediately follows from the definition of weak invariance.
\end{proof}
Hence, if $T(\cdot)$ is continuous from the left 
and the target set $\unitball[y_d]{\delta_0}$ is weakly invariant under $(A,B\Qad)$,
then an iterative procedure is able to find the global optimal solution to 
the time-optimal control problem, provided that the initial value $T_0$
for the minimization of $\delta(\cdot)$
satisfies $\delta(T_0) > 0$. If the latter condition is violated, 
then the procedure has to be restarted with a smaller initial value.
Repeating the steps above will lead to an optimal solution.

\section{The time-optimal control problem}
\label{sec:time_optimal_control}

We introduce a change of variables
to discuss first order necessary optimality conditions for~\eqref{P}.
In particular, we consider optimality conditions in qualified form 
that will be essential for the Newton method (introduced later) to be well-defined.

\subsection{Change of variables}
\label{subsec:change_of_variables}
In order to deal with the variable time horizon of~\eqref{P},
we transform the state equation to the
fixed reference time interval $(0,1)$. 
For $\nu \in \Rplus$ we set $T_\nu(t) = \nu t$ and obtain the transformed state equation
\begin{equation*}
\partial_t y + \nu Ay = \nu \ControlOp u\mbox{,}\quad
y(0) = y_0\mbox{.}
\end{equation*}
We generally abbreviate $I = (0,1)$.
G\aa rding's inequality guarantees that for each pair $(\nu,u) \in \Rplus\times \Q(0,1)$
there exists a unique solution $y \in W(I)$ to the transformed state equation;
see, e.g.,~\cite[Theorem~2, Chapter XVIII, \S3]{Dautray1992}.
Hence, it is justified to introduce the control-to-state mapping 
$S \colon \Rplus\times\Qad(0,1) \rightarrow W(I)$ with
$(\nu,u)\mapsto y = S(\nu,u)$.
The transformed optimal control problem reads as
\begin{equation}\label{Pt}\tag{$\hat{P}$}
\mbox{Minimize~} \nu \mbox{~subject to~}
\left\{
\begin{aligned}
&(\nu,u) \in \Rplus\times\Qad(0,1),\\
&\norm{i_1S(\nu,u) - y_d} \leq \delta_0.
\end{aligned}
\right.
\end{equation}
We emphasize that the problem~\eqref{P} and the transformed 
problem~\eqref{Pt} are equivalent; see \cite[Proposition~4.6]{Bonifacius2017}.

Since there exists at least one feasible control due to
\cref{assumption:existence_feasible_control}, well-posedness of~\eqref{Pt} 
is obtained by the direct method; cf., e.g., \cite[Proposition~4.1]{Bonifacius2017}. 
We note that the optimal solution must fulfill the terminal constraint 
with equality (otherwise, a control with a
shorter time is admissible, while having a smaller objective value).
\begin{proposition}\label{prop:existence}
	Problem~\eqref{Pt} admits a solution $(\bar{\nu}, \bar{u}) \in \Rplus\times\Qad(0,1)$ with
	associated state \(\bar{y} = S(\bar{\nu},\bar{u})\).
	Moreover, \(\norm{\bar{y}(1) - y_d} = \delta_0\) holds.
\end{proposition}

\subsection{First order optimality conditions}
Next, we derive general necessary optimality conditions.	
\begin{lemma}\label{lemma:first_order_optimality}
	Let $(\bar{\nu},\bar{u}) \in \Rplus\times\Qad(0,1)$ be a solution to~\eqref{Pt}.
	Then there exist $\bar{\mu} > 0$ and $\bar{\mu}_0 \in\set{0,1}$
	such that
	\begin{align}
	\int_{0}^{1} \pair{\ControlOp\bar{u}(t) - A\bar{y}(t), \bar{p}(t)}\D{t} &= -\bar{\mu}_0\mbox{,}\label{eq:opt_cond_hamiltonianConstant}\\
	\int_0^{1}\inner{\ControlOp^*\bar{p}(t), u(t) - \bar{u}(t)}_{L^2(\omega)}\D{t} &\geq 0 \quad \text{for all~} u \in \Qad(0,1)\mbox{,}\label{eq:opt_cond_variationalInequality}\\
	\norm{\bar{y}(1) - y_d} &= \delta_0\mbox{,}\label{eq:opt_cond_feasiblity}
	\end{align}
	where the \emph{adjoint state} $\bar{p} \in W(0,1)$ is determined by
	\begin{equation}\label{eq:adjoint_state_equation}
	-\partial_t \bar{p}(t) + \bar{\nu}A^*\bar{p}(t) = 0\mbox{,}
	\quad t \in (0,1) \quad 
	\bar{p}(1) = \bar{\mu}(\bar{y}(1) - y_d)\mbox{.}
	\end{equation}
	If $\bar{\mu}_0 = 1$, then the optimality conditions are called \emph{qualified}.
\end{lemma}
\begin{proof}
	Since the terminal set has finite codimension,
	see, e.g.,~\cite[Definition~4.1.5]{Li1995}, 
	we can argue as in \cite[Theorem~4.13]{Bonifacius2017} to obtain
	$\bar{\mu} > 0$ and $\bar{\mu}_0 \in\set{0,1}$ such that
	\[
	\min_{u\in\Qad(0,1)}\int_{0}^{1}\pair{\ControlOp u - A\bar{y},\bar{p}}\D{t} = \int_{0}^{1}\pair{\ControlOp\bar{u} - A \bar{y},\bar{p}}\D{t} = -\bar{\mu}_0.
	\]
	Now, \eqref{eq:opt_cond_hamiltonianConstant} follows from the second equality
	and the first equality is equivalent to~\eqref{eq:opt_cond_variationalInequality}.
\end{proof}
Last, for the terminal set considered in this article, 
we cite the following criterion from \cite{Bonifacius2017}
that guarantees qualified optimality conditions.
It is worth mentioning that this condition can be checked a priori
without knowing an optimal solution.
\begin{proposition}
	\label{prop:sufficient_cond_qualified_first_order}
	Adapt the assumptions of \cref{prop:T_Lipschitzfromleft_sufficient}.
	Then qualified optimality conditions hold.
\end{proposition}
\begin{proof}
	This follows from \cite[Proposition~5.3]{Bonifacius2017}
	and \cite[Theorem~4.12]{Bonifacius2017}.
\end{proof}

\section{Properties of the minimal distance value function}
\label{sec:minimal_distance_value_function}
We discuss differentiability
of the value function associated with the minimal distance
control problems
that will later be used for a Newton method.
To this end, we first study optimality conditions
and uniqueness of solutions to the minimal distance problems.

\subsection{Minimal distance control problems}
\label{subsec:differentiability_delta}
As in \cref{subsec:change_of_variables} the minimal distance control problem~\eqref{DOPT}
is transformed to the reference time interval $I = (0,1)$.
Moreover, for fixed $\nu \in \Rplus$
we define $\bar{u}(\nu)$ as
\begin{equation}\label{eq:optimaldistance_opt_control}
\bar{u}(\nu) = \argmin_{u \in \Qad(0,1)} \; \norm{i_1S(\nu,u)-y_d}.
\end{equation}
Note that $\bar{u}(\nu)$ is not necessarily unique
and for this reason $\nu \mapsto \bar{u}(\nu)$ is in general a set-valued mapping.
However, the observation $i_1S(\nu,u)$ is unique, 
because in~\eqref{eq:optimaldistance_opt_control}
we can equivalently consider the squared norm
that is strictly convex.
For the following arguments 
we introduce $f \colon \Rplus\times\Q(0,1) \to \R$ defined by
\[
f(\nu, u) = \norm{i_1S(\nu,u)-y_d} - \delta_0.
\]
The minimal distance value function $\delta(\cdot)$ and the functional $f$ are related via
\[
\delta(\nu) = f(\nu,u),\quad u \in \bar{u}(\nu).
\]	
Differentiability of the control to state mapping,
see \cite[Proposition~4.7]{Bonifacius2017},
and the chain rule immediately imply that $f$ is continuously differentiable
for all $\nu \in \Rplus$ such that $\delta(\nu) > -\delta_0$. 
Furthermore, introducing an adjoint state, we have the representation
\[
\partial_\nu f(\nu, u) = \int_{0}^{1}\pair{\ControlOp u - Ay, p}\D{t},
\quad
\partial_u f(\nu, u) = \nu\int_{0}^{1}\inner{\ControlOp^*p, \cdot}_{L^2(\omega)}\D{t},
\]
where $y = S(\nu, u)$ 
and $p \in W(0,1)$ is the associated adjoint state determined by
\begin{equation}
\label{eq:optimaldistance_adjoint_state}
-\partial_t p + \nu A^*p = 0,\quad p(1) = \left(y(1) - y_d\right)/\norm{y(1) - y_d}.
\end{equation}
Note that the adjoint state $p$ is independent of 
the concrete optimal control $u \in \bar{u}(\nu)$,
due to uniqueness of the observation $y(1)$.	
Since both the objective functional in~\eqref{eq:optimaldistance_opt_control} and $\Qad(0,1)$
are convex, the following necessary and sufficient optimality condition holds:
Given $\nu \in \Rplus$ such that $\delta(\nu) > -\delta_0$, a control
$u \in \bar{u}(\nu) \subset \Qad(0,1)$ is optimal for \eqref{eq:optimaldistance_opt_control} if and only if 
\begin{equation}\label{eq:optimaldistance_opt_control_optconditions}
\int_{0}^{1}\inner{\ControlOp^*p, u'-u}_{L^2(\omega)} \geq 0\quad\text{for all } u' \in \Qad(0,1),
\end{equation}
where $p \in W(0,1)$ solves~\eqref{eq:optimaldistance_adjoint_state}
with $y = S(\nu,u)$; 
see, e.g., \cite[Lemma~2.21]{Troeltzsch2010}.
From the variational inequality~\eqref{eq:optimaldistance_opt_control_optconditions}
we deduce that an optimal control $u \in \bar{u}(\nu)$ satisfies
\begin{equation}
\label{eq:optimaldistance_control_projection_formula}
u(t,x) = \begin{cases}
u_a(x) &\text{if~} (\ControlOp^* p)(t,x) > 0\\
u_b(x) &\text{if~} (\ControlOp^* p)(t,x) < 0.
\end{cases}
\end{equation}
Hence, $u$ is bang-bang, if the set where $\ControlOp^* p$ vanishes
has zero measure. 
Indeed, the latter condition ensures uniqueness of the control.	
\begin{proposition}
	\label{prop:optimaldistance_opt_control_unique}
	Let $\nu \in \Rplus$ such that $\delta(\nu) > -\delta_0$ and $u \in \bar{u}(\nu)$. 
	Moreover, suppose that the associated adjoint state $p$ determined by~\eqref{eq:optimaldistance_adjoint_state}
	satisfies
	\begin{equation}
	\label{eq:optimaldistance_adjoint_measure_zeroset}
	\abs{\set{(t,x) \in I\times\omega \colon (\ControlOp^* p)(t,x) = 0}} = 0,
	\end{equation}
	where $\abs{\cdot}$ denotes the measure associated with $I\times\omega$.
	Then $u$ is bang-bang and $\bar{u}(\nu)$ is a singleton.		
\end{proposition}
\begin{proof}
	Due to uniqueness of the observation $y(1)$, 
	the associated adjoint state $p$ is unique.
	Hence, from~\eqref{eq:optimaldistance_control_projection_formula}
	and \eqref{eq:optimaldistance_adjoint_measure_zeroset}
	we conclude that $u$ is bang-bang and unique.
\end{proof}

Condition~\eqref{eq:optimaldistance_adjoint_measure_zeroset} can be
deduced from a unique continuation property.
Let $p$ denote the adjoint state with terminal value $p_1 \in H$. 
The system satisfies the \emph{unique continuation property}
\begin{equation}
\label{eq:unique_continuation_property}
\text{if~}B^*p = 0\text{~on some~}\Lambda\subset I\times\omega
\text{~with~}\abs{\Lambda} \neq 0,
\text{~then~} p = 0.
\end{equation}
\begin{proposition}
	\label{prop:optimaldistance_unique_continuation}
	If the unique continuation property~\eqref{eq:unique_continuation_property} is satisfied, 
	then~\eqref{eq:optimaldistance_adjoint_measure_zeroset} holds.
	In particular, $\bar{u}(\nu)$ is a singleton.		
\end{proposition}
\begin{proof}
	Assume that condition~\eqref{eq:optimaldistance_adjoint_measure_zeroset}
	is violated. Then there exists a subset $\Lambda \subset I\times\omega$
	with nontrivial measure
	such that $B^*p = 0$ on $\Lambda$.
	From the unique continuation property~\eqref{eq:unique_continuation_property}
	we deduce $p = 0$.
	This contradicts $p(1) = \left(y(1) - y_d\right)/\norm{y(1) - y_d} \neq 0$.
\end{proof}

\begin{remark}
	\label{remark:unique_continuation}
	The unique continuation property (also referred to as backward uniqueness property)
	is guaranteed to hold in the following situations.
	\begin{enumerate}[(i)]
		\item In the case of purely time-dependent controls, i.e.\
		$\ControlOp u(t) = \sum_{i=1}^{\parameterControlDim} e_i u_i(t)$, 
		$u \in L^2(I; \R^{\parameterControlDim})$, for $e_i \in V^*$,
		the unique continuation property is equivalent to normality of $(A,B)$;
		see and \cite[Theorem~11.2.1, Definition~6.1.1]{Tucsnak2009}.	
		A system $(A,B)$ is called normal, if
		$(A,\ControlOp_i)$ is approximately controllable for all $i = 1,2,\ldots,\parameterControlDim$;
		cf.\ also \cite[Section~II.16]{Hermes1969} or \cite[Section~III.3]{Macki1982}.
		\item For the linear heat-equation on a bounded domain
		with a distributed control
		acting on an open subset of the spatial domain,
		the unique continuation property is known to hold; 
		see \cite[Theorem~4.7.12]{Fattorini2005} using \cite[Theorem~1.1]{Han1994}.
	\end{enumerate}	
\end{remark}

\subsection{Differentiability of $\delta(\cdot)$}
Next, we present the central differentiability result of this section.
After its proof, we discuss specific situations where the directional
derivative of $\delta(\cdot)$ can be strengthened to a classical derivative.

\begin{theorem}\label{lemma:optimaldistance_value_function_differentiable}
	Let $\nu \in \Rplus$ such that $\delta(\nu) > -\delta_0$.
	Then the value function $\delta(\cdot)$ is directionally differentiable at $\nu$
	and the expression
	\begin{equation}\label{eq:optimaldistance_value_function_derivative}
	\mathrm{d}^{\pm}\delta(\nu) = \min_{u \in \bar{u}(\nu)}\pm\int_0^1\pair{\ControlOp u - Ay,p}\D{t},
	\end{equation}
	holds, where $p \in W(0,1)$ satisfies
	\[
	-\partial_t p + \nu A^*p = 0,\quad 
	p(1) = \left(y(1) - y_d\right)/\norm{y(1) - y_d},
	\]
	and $y = S(\nu,u)$.
	If additionally the value of the integral in \eqref{eq:optimaldistance_value_function_derivative}
	is independent of the concrete minimizer $u \in \bar{u}(\nu)$,
	then $\delta(\cdot)$ is continuously differentiable at $\nu$.
	Here, $\mathrm{d}^{+}$ and $\mathrm{d}^{-}$
	denote the right and left directional derivatives.
\end{theorem}
For the proof we require
\begin{proposition}\label{prop:state_finaltime_cont_transformation_uniform}
	Let $\nu \in \Rplus$. Then
	\[
	\lim_{n \to \infty}\sup_{u \in \Qad(0,1)} \norm{i_1S(\nu_n,u) - i_1S(\nu,u)} = 0
	\]
	for all sequences $\nu_n \in \Rplus$ such that $\nu_n \to \nu$.
\end{proposition}
\begin{proof}
	Set $y_n = y(\nu_n)$ and $y = y(\nu)$. 
	Then the difference $w = y - y_n$ satisfies
	\[
	\partial_t w + \nu A w = (\nu-\nu_n)\left(-A y_n + Bu\right), \quad w(0) = 0.
	\]
	Hence, the assertion follows by
	standard energy estimates as well as
	the embedding $\MPRHilbert{I}{V^*}{V} \embedding C([0,1]; H)$.
\end{proof}
\begin{proof}[Proof of \cref{lemma:optimaldistance_value_function_differentiable}]
	Let $\nu \in \Rplus$ and $\tau_n \in \R$ such that $\tau_n \to 0$.
	Set $\nu_n = \nu + \tau_n$ and $u_n \in \bar{u}(\nu_n)$. 
	Due to boundedness of $u_n \in \Qad(0,1)$, there exists a subsequence denoted in the same way
	such that $u_n \rightharpoonup u$ in $L^s(I\times\omega)$ for some $s > 2$ 
	as $n \to \infty$ with $u \in \Qad(0,1)$.
	Let $\tilde{u} \in \Qad(0,1)$ denote a minimizer of \eqref{eq:optimaldistance_opt_control}.
	Affine linearity of $u \mapsto S(\nu,u)$ for fixed $\nu$,
	weak lower semi continuity of $\norm{\cdot}$,
	and optimality of $(\nu_n,u_n)$ imply
	\begin{align*}
	f(\nu, u) &\leq \liminf_{n\to\infty} f(\nu,u_n)
	\leq \limsup_{n\to\infty}f(\nu_n,u_n) + \limsup_{n\to\infty} \left[f(\nu,u_n) - f(\nu_n,u_n)\right]\\
	&\leq \limsup_{n\to\infty}f(\nu_n,\tilde{u}) = f(\nu,\tilde{u}),
	\end{align*}
	where we have used \cref{prop:state_finaltime_cont_transformation_uniform} in the second last step.
	Hence, the weak limit $u$ is also a minimizer of~\eqref{eq:optimaldistance_opt_control},
	i.e.\ $u \in \bar{u}(\nu)$.

	Optimality of the tuples $(\nu,u)$ and $(\nu_n, u_n)$ leads to
	\begin{align*}
	f(\nu_n, u_n) - f(\nu, u_n)
	&\leq f(\nu_n, u_n) - f(\nu, u_n) + f(\nu, u_n) - f(\nu,u)\\
	&= \delta(\nu_n)- \delta(\nu)
	\leq f(\nu_n, u) - f(\nu,u).
	\end{align*}
	Without restriction suppose that $\tau_n > 0$ for all $n \in \N$. 
	Dividing the above chain of inequalities by $\tau_n$, we infer that
	\begin{equation}\label{eq:optimaldistance_value_function_derivative_P5}
	\tau_n^{-1}\left[f(\nu_n, u_n) - f(\nu, u_n)\right]			
	\leq \tau_n^{-1}\left[\delta(\nu_n) - \delta(\nu)\right] 
	\leq \tau_n^{-1}\left[f(\nu_n, u) - f(\nu,u)\right].
	\end{equation}
	The right-hand side of \eqref{eq:optimaldistance_value_function_derivative_P5}
	converges to $\partial_\nu f(\nu,u)$ due to differentiability of the control-to-state mapping.
	Concerning the left-hand side,
	we first observe that
	\[
	\tau_n^{-1}\left[f(\nu_n, u_n) - f(\nu, u_n)\right]
	= \partial_\nu f(\nu + \theta_n, u_n)
	= \int_{0}^{1}\pair{\ControlOp u_n - Ay_n,p_n}\D{t}
	\]
	with $\theta_n \to 0$, $y_n = S(\nu + \theta_n, u_n)$,
	and $p_n$ the associated adjoint state with terminal value $y_n(1)-y_d$.
	Convergence of $\nu_n \to \nu$, weak convergence of $u_n \rightharpoonup u$, and compactness of $(\nu,u) \mapsto S(\nu,u)$ from
	$\Rplus\times L^s(I\times\omega)$ to $C([0,1]; H)$, 
	see \cite[Proposition~A.19]{Bonifacius2018},
	yields $p_n \to p$ in $W(0,1)$. Hence
	\[
	\lim_{n\to \infty}\int_{0}^{1}\pair{\ControlOp u_n - Ay_n,p_n}\D{t}
	= \int_{0}^{1}\pair{\ControlOp u - Ay,p}\D{t}.
	\]
	In summary, this proves
	\begin{equation}
	\label{eq:optimaldistance_value_function_differentiable_P5}
	\lim_{n \to \infty}\tau_n^{-1}\left[\delta(\nu_n) - \delta(\nu)\right] = \int_0^1\pair{\ControlOp u - Ay,p}\D{t}.
	\end{equation}
	We have to argue that the limit is independent of the chosen subsequence.
	To this end, we first observe that
	\[
	\delta(\nu_n)- \delta(\nu)
	\leq f(\nu_n,\tilde{u}) - f(\nu, \tilde{u})
	\]
	for any minimizer $\tilde{u}$ of \eqref{eq:optimaldistance_opt_control}.
	Hence, dividing the inequality above by $\tau_n$ and passing to the limit 
	implies the additional estimate
	\begin{equation}\label{eq:optimaldistance_value_function_differentiable_P8}
	\lim_{n \to \infty}\tau_n^{-1}\left[\delta(\nu_n) - \delta(\nu)\right] 
	\leq \int_0^1\pair{\ControlOp \tilde{u} - A\tilde{y},p}\D{t},
	\end{equation}
	where $\tilde{y} = S(\nu,\tilde{u})$.
	Recall that the adjoint state $p$ is unique due to uniqueness of the observation.
	Let $u_{n}'$ denote another subsequence of $u_n$ with weak limit $u' \in \Qad(0,1)$
	and associated times $\nu_n' = \nu + \tau_n'$.
	Repeating the arguments above we obtain
	\begin{equation}\label{eq:optimaldistance_value_function_differentiable_P9}
	\int_0^1\pair{\ControlOp u' - Ay',p'}\D{t}
	= \lim_{n' \to \infty}\tau_{n}'^{-1}\left[\delta(\nu_{n}') - \delta(\nu)\right]
	\leq \int_0^1\pair{\ControlOp \tilde{u} - A\tilde{y},p}\D{t}
	\end{equation}
	for any minimizer $\tilde{u}$ of \eqref{eq:optimaldistance_opt_control}.
	Now, combining~\eqref{eq:optimaldistance_value_function_differentiable_P5},
	as well as~\eqref{eq:optimaldistance_value_function_differentiable_P8} with $\tilde{u} = u'$,
	and~\eqref{eq:optimaldistance_value_function_differentiable_P9} with $\tilde{u} = u$
	yields
	\begin{align*}
	\int_0^1\pair{\ControlOp u - Ay,p}\D{t}
	&= \lim_{n \to \infty}\tau_n^{-1}\left[\delta(\nu_n) - \delta(\nu)\right]
	\leq \int_0^1\pair{\ControlOp u' - Ay',p}\D{t}\\
	&= \lim_{n' \to \infty}\tau_{n}'^{-1}\left[\delta(\nu_{n}') - \delta(\nu)\right]
	\leq \int_0^1\pair{\ControlOp u - Ay,p}\D{t}.
	\end{align*}
	Hence, equality must hold and we conclude that the limit is 
	independent of the chosen subsequence.
	Taking the infimum in the inequalities above implies
	\[
	\mathrm{d}^{\pm}\delta(\nu) = \inf_{u \in \bar{u}(\nu)}\pm\int_0^1\pair{\ControlOp u - Ay,p}\D{t}.
	\]
	By standard arguments we can show 
	that the infimum exists and
	we conclude~\eqref{eq:optimaldistance_value_function_derivative}.
	
	Clearly, if the integral expression in~\eqref{eq:optimaldistance_value_function_derivative}
	is independent of $u$, then $\delta(\cdot)$ is differentiable at $\nu$.
	To show continuity of $\delta'(\cdot)$, let $\nu_n \in \Rplus$ with $\nu_n \to \nu$.
	Moreover, let $u_n \in \bar{u}(\nu_n)$ such that $u_n$
	minimizes the expression~\eqref{eq:optimaldistance_value_function_derivative}
	for $\nu = \nu_n$.
	As in the beginning of the proof, there exists a subsequence converging weakly to $u \in \Qad(0,1)$
	that is a minimizer of \eqref{eq:optimaldistance_opt_control}.
	Compactness of the control-to-state mapping 
	from $\Rplus\times L^s(I\times\omega)$ to $C([0,1]; H)$, 
	see \cite[Proposition~A.19]{Bonifacius2018},
	as before leads to
	\[
	\lim_{n \to \infty} \mathrm{d}^{\pm}\delta(\nu_n) = \lim_{n \to \infty} \pm\int_0^1\pair{\ControlOp u_n - Ay_n,p_n}\D{t}
	= \pm\int_0^1\pair{\ControlOp u - Ay,p}\D{t}
	\]
	where $y_n = S(\nu_n, u_n)$ and $y = S(\nu,u)$ with $p_n$ and $p$ 
	denoting the associated adjoint states.
	Hence, we conclude that $\delta'(\cdot)$ is continuous.
\end{proof}	
If $\bar{u}(\nu)$ is a singleton,
which can be guaranteed under the unique continuation property (see \cref{prop:optimaldistance_unique_continuation}), 
then we immediately deduce that $\delta(\cdot)$ is continuously differentiable.	
\begin{corollary}
	\label{cor:optimaldistance_value_function_backwards_uniqueness}
	If the unique continuation property~\eqref{eq:unique_continuation_property} holds,
	the integral expression in~\eqref{eq:optimaldistance_value_function_derivative}
	is independent of $u \in \bar{u}(\nu)$.
	In particular, $\delta(\cdot)$ is continuously differentiable.
\end{corollary}	
Moreover, in the case of purely time-dependent controls,
the expression for the derivative is independent
of the concrete minimizer $u \in \bar{u}(\nu)$,
even for multiple optimal controls.
\begin{proposition}
	\label{cor:optimaldistance_value_function_parameter_control}
	In the case of purely time-dependent controls
	(i.e.\ $\omega = \set{1,2,\ldots,\parameterControlDim}$ 
	equipped with the counting measure in \cref{assumptionControl}),
	the integral expression in~\eqref{eq:optimaldistance_value_function_derivative}
	is independent of $u \in \bar{u}(\nu)$.
	In particular, $\delta(\cdot)$ is continuously differentiable.
\end{proposition}
\begin{proof}
	We consider the splitting
	\[
	\int_0^1\pair{\ControlOp u - Ay,p}\D{t} = 
	\int_0^1\pair{\ControlOp u,p} - \pair{Ay_1,p} - \pair{Ay_2,p}\D{t}			
	\]
	with $y_1 = S(\nu, 0)$ and $y_2 = y - y_1$.
	Recall that the adjoint state $p$ is independent of $u$,
	due to uniqueness of the observation.
	Hence, the optimality condition~\eqref{eq:optimaldistance_control_projection_formula}
	for $u \in \bar{u}(\nu)$ implies
	that the first summand is independent of $u$.
	Moreover, the second summand is independent of $u$, 
	because $y_1$ depends on the initial state $y_0$ and the time $\nu$, only.
	For the remaining summand,
	the variation of constants formula yields
	\begin{align*}
	\pair{Ay_{2}(t),p(t)} &= \nu\pair{A\int_{0}^{t}\semigroup{-\nu(t-s)A}Bu(s) \D{s}, \semigroup{-\nu(1-t)A^*} p(1)}\\
	&= \nu\int_{0}^{t}\pair{Bu(s), A^*\semigroup{-\nu(t-s)A^*}\semigroup{-\nu(1-t)A^*} p(1)}\D{s}
	= \int_{0}^{t}\pair{Bu(s), \nu A^*p(s)}\D{s},
	\end{align*}
	where we have used the identity $(\semigroup{-\cdot A})^* = \semigroup{-\cdot A^*}$, see \cite[Corollary~1.10.6]{Pazy1983},
	the fact that the semigroup commutes with its generator, see \cite[Theorem~1.2.4]{Pazy1983},
	and the semigroup property.
	Hence, Fubini's theorem and the definition of $p$ imply
	\begin{align*}
	\int_{0}^{1}\pair{Ay_{2}(t),p(t)}\D{t} &= 
	\int_{0}^{1}\int_{0}^{1}\mathds{1}_{[0,t]}(s)\pair{Bu(s), \nu A^*p(s)}\D{s}\D{t}\\
	&= 
	\int_{0}^{1}\int_{0}^{1}\mathds{1}_{[s,1]}(t)\pair{Bu(s), \nu A^*p(s)}\D{t}\D{s}\\
	&= \int_{0}^{1} (1-s)\pair{Bu(s), \nu A^*p(s)}\D{s}
	= \int_{0}^{1} (1-s)\pair{u(s), B^*\partial_t p(s)}\D{s}.
	\end{align*}
	Since $\omega$ is discrete, we can identify $\left(\ControlOp^*p\right)(i) = (\ControlOp^*p)_i$
	for $i \in \set{1,2,\ldots,\parameterControlDim}$.
	If $(\ControlOp^*p)_i$ vanishes on a set with nonzero measure for some $i \in \set{1,2,\ldots,\parameterControlDim}$,
	then it has to vanish on $(0,1)$ due to analyticity of the semigroup
	generated by $-A^*$.		
	Thus, $(B^*\partial_t p)_i = \partial_t (B^*p)_i = 0$ on $(0,1)$.
	Due to uniqueness of the adjoint state $p$
	and the fact that only those components of $u$ are not uniquely determined
	where $B^*p$ vanishes (see optimality condition~\eqref{eq:optimaldistance_control_projection_formula})
	we conclude that the above expression is independent of $u$.
	Last, the second assertion follows from the first and \cref{lemma:optimaldistance_value_function_differentiable}.
\end{proof}

\subsection{Lipschitz continuity of $\delta'(\cdot)$}
\label{subsec:local_Lipschitz_continuity}
Last, we consider a sufficient condition for Lipschitz continuity
of $\delta'(\cdot)$, which in turn guarantees fast local convergence
of the Newton method.
Let $\nu \in \Rplus$, $u \in \bar{u}(\nu) \subset \Qad(0,1)$, and let $p$ denote the corresponding adjoint state.
We say that the structural assumption holds at $\bar{u}(\nu)$, if there exists a $C > 0$ such that
\begin{equation}\label{eq:assumption_adjoint_bb}
\abs{\set{(t,x) \in I\times\omega \colon -\varepsilon \leq (\ControlOp^*p)(t,x) \leq \varepsilon }} \leq C \varepsilon
\end{equation}
for all $\varepsilon > 0$. 
Since~\eqref{eq:assumption_adjoint_bb} implies that $\bar{u}(\nu)$
is a singleton, see \cref{prop:optimaldistance_opt_control_unique},
it is justified to say that \eqref{eq:assumption_adjoint_bb} holds at $\bar{u}(\nu)$.
\begin{proposition}\label{prop:derivativeF_lowerbound_L1}
	Let $\nu \in \Rplus$ and suppose that~\eqref{eq:assumption_adjoint_bb}
	holds at $\set{u} = \bar{u}(\nu)$. Then
	\begin{equation}
	\label{eq:derivativeF_lowerbound_L1}
	\partial_u f(\nu,u)(u'-u) \geq \nu c_0 \norm{u' - u}_{L^1(I\times\omega)}^2
	\quad\text{for all~} u' \in \Qad(0,1),
	\end{equation}
	where $c_0 = (2\norm{u_b-u_a}_{L^\infty(\omega)}C)^{-1}$.
\end{proposition}
\begin{proof}
	The proof can be obtained along the lines of \cite[Proposition~2.7]{Casas2016}.
\end{proof}
For the following considerations, we assume that
the adjoint states $p(\nu_1)$ and $p(\nu_2)$ associated with 
the time transformations $\nu_1$ and $\nu_2$ 
and the states $i_1S(\nu_1, u)$
and $i_1S(\nu_2, u)$
satisfy
\begin{equation}
\label{eq:stability_estimate_adjoints}
\norm{\ControlOp^*\left(p(\nu_1) - p(\nu_2)\right)}_{L^\infty(I\times\omega)}
\leq c \abs{\nu_1 - \nu_2}
\end{equation}
for all $\nu_1, \nu_2 \in [\nu_{\text{min}}, \nu_{\text{max}}]$ 
and all $u \in \Qad(0,1)$, 
where $0 < \nu_{\text{min}} < \nu_{\text{max}}$ are constants.
The stability estimate~\eqref{eq:stability_estimate_adjoints} holds in case of purely-time dependent controls
under the general conditions of this article.
Moreover, the estimate can be shown in case of a distributed control for fairly general elliptic operators
and spatial domains; see, e.g., \cite[Proposition~A.3]{Bonifacius2018b}.

\begin{proposition}\label{prop:lipschitz_continuity_optimalcontrol_nu}
	Suppose that \eqref{eq:stability_estimate_adjoints} is valid
	and let $\bar{\nu} \in \Rplus$ with $\set{\bar{u}} = \bar{u}(\bar{\nu})$.
	If \eqref{eq:derivativeF_lowerbound_L1} holds at $(\bar{\nu},\bar{u})$ 
	for some constant $c_0 > 0$, then there is $\delta > 0$ such that
	\[
	\norm{u -\bar{u}}_{L^1(I\times\omega)} \leq c\abs{\nu-\bar{\nu}}\quad\text{for all~} \nu \in \Rplus, 
	\abs{\nu - \bar{\nu}} \leq \delta,\text{~and~}
	u \in \bar{u}(\nu), 
	\]
	with $c > 0$ a constant independent of $\nu$ and $u$.
\end{proposition}
\begin{proof}
	Let $u \in \bar{u}(\nu)$
	and let $p(\nu,u)$ denote the associated adjoint state.
	Employing \cref{prop:derivativeF_lowerbound_L1} with $u' = u$
	and the first order necessary optimality condition~\eqref{eq:optimaldistance_opt_control_optconditions} for $u$ yield
	\begin{multline*}
	\bar{\nu} c_0 \norm{u - \bar{u}}_{L^1(I\times\omega)}^2
	\leq \partial_u f(\bar{\nu},\bar{u})(u-\bar{u})
	\leq \partial_u f(\bar{\nu},\bar{u})(u-\bar{u}) - \bar{\nu}\inner{\ControlOp^*p(\nu,u),u-\bar{u}}_{L^2(I\times\omega)}\\
	= \bar{\nu}\inner{\ControlOp^*\left(\bar{p}-p(\bar{\nu},u)\right),u-\bar{u}}_{L^2(I\times\omega)} + \bar{\nu}\inner{\ControlOp^*\left(p(\bar{\nu},u) - p(\nu,u)\right),u-\bar{u}}_{L^2(I\times\omega)},
	\end{multline*}
	where $p(\bar{\nu},u)$ denotes the adjoint state
	associated with $\bar{\nu}$ and $u$.
	Concerning the first term on the right-hand side we observe
	\[
	\bar{\nu}\inner{\ControlOp^*\left(\bar{p}-p(\bar{\nu},u)\right),u-\bar{u}}_{L^2(I\times\omega)}
	= -\bar{\nu}\norm{i_1\left(\partial_t + \bar{\nu}A\right)^{-1}\ControlOp(u-\bar{u})}^2 \leq 0,
	\]
	where $\left(\partial_t + \bar{\nu}A\right)^{-1}$
	denotes the solution operator to the linear parabolic state equation with zero initial value. 
	Thus, H\"older's inequality implies
	\[
	c_0 \norm{u - \bar{u}}_{L^1(I\times\omega)}
	\leq \norm{\ControlOp^*\left(p(\bar{\nu}) - p(\nu)\right)}_{L^\infty(I\times\omega)}.
	\]
	Finally, we apply the stability estimate~\eqref{eq:stability_estimate_adjoints} to conclude the proof.
\end{proof}
Applying \cref{prop:lipschitz_continuity_optimalcontrol_nu} twice,
we immediately infer the following Lipschitz type estimate
\begin{corollary}
	\label{prop:optcontr_local_Lipschitz}
	There are $\delta > 0$ and $c > 0$ such that
	\begin{equation}
	\label{eq:optcontrol_Lipschitztype_estimate}
	\norm{u_1 - u_2}_{L^1(I\times\omega)} \leq c\abs{\nu_1 - \nu_2}
	\quad\text{for all~}u_1 \in \bar{u}(\nu_1)\text{~and~}
	u_2 \in \bar{u}(\nu_2),
	\end{equation}
	and all $\nu_1 \in [\bar{\nu} - \delta, \bar{\nu}]$
	and $\nu_2 \in [\bar{\nu}, \bar{\nu} + \delta]$.
\end{corollary}	
Moreover, if $\ControlOp \in \mathcal{L}(L^1(\omega),H)$
then the control-to-state mapping is continuous from 
$L^1(I\times\omega)$ to $C([0,1]; H)$
and we infer the following result.
\begin{corollary}
	\label{cor:delta_local_Lipschitz}
	If $\ControlOp \in \mathcal{L}(L^1(\omega),H)$,
	then there are $\delta > 0$ and $c > 0$ such that
	\[
	\abs{\delta'(\nu_1) - \delta'(\nu_2)} \leq c\abs{\nu_1 - \nu_2}
	\]
	for all $\nu_1 \in [\bar{\nu} - \delta, \bar{\nu}]$
	and $\nu_2 \in [\bar{\nu}, \bar{\nu} + \delta]$.
\end{corollary}

\section{Algorithm}
\label{sec:algorithm}
We now turn to the algorithmic solution
of~\eqref{P}.	
Throughout the rest of this article we assume that $T(\cdot)$ is left-continuous.
In view of \cref{thm:equivalence_topt_dopt}, we are interested in finding 
a root of
the value function $\delta(\cdot)$ in order to solve the time-optimal control problem \eqref{P}. 
This will generally lead to a bi-level optimization problem: The outer loop finds the 
optimal $T$ and the inner loop determines for each given $T$ a 
control such that the associated state has a minimal distance to the target set.
It is worth mentioning that this procedure will
find a global solution to \eqref{P} provided that 
we initiate the outer optimization with a time smaller than the optimal one.

\subsection{Newton method for the outer minimization}

To find a root of the value function, we apply the Newton method.
As this requires $\delta(\cdot)$ to be continuously differentiable,
we require the following assumption.
Recall that \cref{assumption:optimaldistance_value_function_derivative_control_independent}
automatically holds in the case of purely time-dependent controls
and for the linear heat-equation on a bounded domain
with distributed control; see
\cref{cor:optimaldistance_value_function_backwards_uniqueness,cor:optimaldistance_value_function_parameter_control}.
\begin{assumption}
	\label{assumption:optimaldistance_value_function_derivative_control_independent}
	Suppose that the integral expression in~\eqref{eq:optimaldistance_value_function_derivative}
	is independent of the concrete minimizer $u \in \bar{u}(\nu)$
	for all $\nu \in \Rplus$ with $\delta(\nu) > -\delta_0$.
\end{assumption}	
For well-posedness of the method, we have to guarantee that $\delta'(\bar{\nu}) \neq 0$.
The following result underlines the practical relevance of qualified optimality conditions 
for~\eqref{Pt} in the context of its algorithmic solution.
\begin{proposition}\label{prop:equiv_qualified_optcond_deltaprime_nonzero}
	Let $(\bar{\nu},\bar{u}) \in \Rplus\times\Qad(0,1)$ be a solution to~\eqref{Pt}.
	The first order optimality conditions of \cref{lemma:first_order_optimality} 
	hold in qualified form
	if and only if $\delta'(\bar{\nu}) \neq 0$.
\end{proposition}
\begin{proof}
	According to the general form of the optimality conditions of \cref{lemma:first_order_optimality}
	there exist $\bar{\mu} > 0$ and $\bar{\mu}_0 \in \set{0,1}$ such that
	\[
	\int_0^1\pair{\ControlOp \bar{u} - A\bar{y},\bar{p}}\D{t} = -\bar{\mu}_0,
	\]
	where $\bar{p} \in W(0,1)$ is the adjoint state with
	terminal value $\bar{\mu}\left(\bar{y}(1)- y_d\right)$.
	Hence, linearity of the expression above and~\eqref{eq:optimaldistance_value_function_derivative}
	imply $\delta'(\bar{\nu}) = -\bar{\mu}^{-1}\bar{\mu}_0\norm{\bar{y}(1)-y_d}^{-1}$.
	Thus, qualified optimality conditions (i.e.\ $\bar{\mu}_0 = 1$) hold
	if and only if $\delta'(\bar{\nu}) \neq 0$.
\end{proof}
The resulting Newton method is summarized in \cref{alg:optimaldistance_outer_newton}.
By means of \cref{lemma:optimaldistance_value_function_differentiable} 
and well-known properties of the Newton method, see, e.g., \cite[Theorem~11.2]{Nocedal2006},
we obtain the following convergence result.
\begin{proposition}
	\label{prop:newton_method_convergence_superlinear}
	Let $\bar{\nu} \in \Rplus$ and suppose that 
	\cref{assumption:optimaldistance_value_function_derivative_control_independent} holds.
	If $\delta'(\bar{\nu}) \neq 0$, then the sequence $\nu_n$ generated by 
	\cref{alg:optimaldistance_outer_newton} converges locally q-superlinearly to $\bar{\nu}$.
\end{proposition}

\begin{algorithm}
	Choose $\nu_0 > 0$\;		
	\For{$n = 0$ \KwTo $n_{\max}$}{
		Calculate $u \in \bar{u}(\nu_n)$ using \cref{alg:optimaldistance_inner_gcg} and $y = S(\nu_n,u)$\;
		\If{$\delta(\nu_n) < \algtol{tol}$}{ 
			\Return\; 
		}			
		Evaluate $\delta'(\nu_n)$ using \eqref{eq:optimaldistance_value_function_derivative}\;
		Set $\nu_{n+1} = \nu_n - \delta(\nu_n) \delta'(\nu_n)^{-1}$\;
	}
	\caption{Newton method for solution of minimal distance problem (outer loop)}
	\label{alg:optimaldistance_outer_newton}
\end{algorithm}

If we in addition assume that the control operator is bounded from $L^1$ into $H$, 
then the variation of constants formula implies that
the control-to-state mapping is linear and continuous 
from $L^1(I\times\omega)$ to $C([0,1]; H)$ for any fixed $\nu \in \Rplus$.
Hence, if the structural assumption~\eqref{eq:assumption_adjoint_bb}
on the adjoint state holds, we immediately obtain the following
fast convergence result.
\begin{proposition}
	\label{prop:newton_method_convergence_quadratic}
	Let $\bar{\nu} \in \Rplus$ and suppose that 
	\cref{assumption:optimaldistance_value_function_derivative_control_independent} holds.
	Moreover, assume 
	that~\eqref{eq:assumption_adjoint_bb} holds at $\bar{u}(\bar{\nu})$
	and that $\ControlOp \in \mathcal{L}(L^1(\omega), H)$.
	If $\delta'(\bar{\nu}) \neq 0$, then the sequence $\nu_n$ generated by 
	\cref{alg:optimaldistance_outer_newton} 
	converges locally q-quadratically to $\bar{\nu}$.
\end{proposition}
\begin{proof}
	First, \cref{prop:newton_method_convergence_superlinear} guarantees
	q-linear convergence of the sequence $\nu_n$.
	The improved convergence rate follows from 
	Lipschitz continuity of $\delta'(\cdot)$, see \cref{cor:delta_local_Lipschitz},
	and well-known properties
	of the Newton method; see, e.g., \cite[Theorem~11.2]{Nocedal2006}.
	Note that the Lipschitz type estimate of \cref{cor:delta_local_Lipschitz} is sufficient
	for the proof of \cite[Theorem~11.2]{Nocedal2006}.
\end{proof}

\begin{remark}
	For convenience we summarize that under \cref{assumption:optimaldistance_value_function_derivative_control_independent}
	(which implies that $\delta(\cdot)$ is continuously differentiable)
	and if qualified optimality conditions hold for~\eqref{Pt}
	(which implies that $\delta'(\cdot)$ is nonzero near the optimal solution),
	the Newton method for finding a root of $\delta(\cdot)$ is well-defined.
	If in addition, $T(\cdot)$ is left-continuous, 
	then the root of $\delta(\cdot)$ is the optimal time for the time-optimal control problem~\eqref{Pt}.
\end{remark}

\subsection{Conditional gradient method for the inner optimization}

For the algorithmic solution of the inner problem, i.e.\ the determination 
of $\bar{u}(\nu)$ in~\eqref{eq:optimaldistance_opt_control},
we employ the conditional gradient method; see, e.g., \cite{Dunn1980}. 
We abbreviate
\[
f(u) = \norm{i_1S(\nu,u)-y_d}
\]
neglecting the $\nu$ dependence for a moment. 
Clearly, we are interested in minimizing $f$ over $\Qad(0,1)$. 
As in \cref{subsec:differentiability_delta},
we have
\[
f'(u)^* = \nu\ControlOp^*p,
\]
where $p \in W(0,1)$ solves \eqref{eq:optimaldistance_adjoint_state} with $y = S(\nu,u)$. 
Given $u_n \in \Qad(0,1)$, we take
\begin{equation}
\label{eq:optimaldistance_gcg_dq}
u_{n+1/2} = \begin{cases}
u_a, &\text{if } \ControlOp^*p_n > 0,\\
u_b, &\text{if } \ControlOp^*p_n < 0,\\
(u_a + u_b)/2, &\text{else},
\end{cases}
\end{equation}
almost everywhere.
This choice guarantees that
\[
f'(u_n)(u_{n+1/2} - u_n) = \min_{u \in \Qad(0,1)} f'(u_n)(u - u_n).
\]
The next iterate $u_{n+1}$ is defined by the 
optimal convex combination of $u_n$ and $u_{n+1/2}$. Precisely,
we take $u_{n+1} = (1-\lambda^*)u_n + \lambda^*u_{n+1/2}$ with
\begin{equation}
\label{eq:optimaldistance_gcg_lambda}
\lambda^* = \argmin_{0\leq \lambda \leq 1} f((1-\lambda)u_n + \lambda u_{n+1/2}).
\end{equation}
This expression can be analytically determined, employing the fact that $u \mapsto S(\nu,u)$ is affine linear.	
Using the convexity of $f$ and the definition of $u_{n+1/2}$, we immediately derive the following a posteriori error estimator
\[
0 \leq f(u_n)-f(\bar{u}) \leq f'(u_n)(u_n - \bar{u}) 
\leq \max_{u \in \Qad(0,1)} f'(u_n)(u_n - u) 
= f'(u_n)(u_n - u_{n+1/2}).
\]
The expression on the right-hand side can be efficiently evaluated using the adjoint representation 
and serves as a termination criterion for the conditional gradient method. 
The algorithm for the inner optimization is summarized in \cref{alg:optimaldistance_inner_gcg}.

\begin{algorithm}
	Let $\nu > 0$ be given. Choose $u_0 \in \Qad(0,1)$\;		
	\For{$n = 0$ \KwTo $n_{\max}$}{
		Calculate $y_n = S(\nu,u_n)$ and $p_n$\;
		Choose $u_{n+1/2}$ as in \eqref{eq:optimaldistance_gcg_dq}\;
		\If{$f'(u_n)(u_n - u_{n+1/2}) < \algtol{tol}$}{ 
			\Return\; 
		}		
		Calculate $\lambda^*$ as in \eqref{eq:optimaldistance_gcg_lambda}\;
		Set $u_{n+1} = (1-\lambda^*)u_n + \lambda^* u_{n+1/2}$\;
	}
	\caption{Conditional gradient method for solution of \eqref{eq:optimaldistance_opt_control}}
	\label{alg:optimaldistance_inner_gcg}
\end{algorithm}	

The conditional gradient method has the following convergence properties.
\begin{proposition}
	Let $(u_n)_n$ be a sequence generated by the conditional gradient method. 
	Then $f(u_n)$ decreases monotonically and
	\[
	0 \leq f(u_n) - f(\bar{u}) \leq 
	\frac{f(u_0) - f(\bar{u})}{1 + c n},\quad n \geq 0,
	\]
	with a constant $c$ exclusively depending on the Lipschitz constant of $f'$ on $\Qad(0,1)$, the initial residuum, and $\Qad$.	
\end{proposition}
\begin{proof}
	This follows from \cite[Theorem~3.1~(i)]{Dunn1980}, since both $f$ and $\Qad(0,1)$ are convex.
\end{proof}
If the control operator $\ControlOp$ defines a bounded operator from $L^1(\omega)$ to $H$, then under the structural assumption~\eqref{eq:assumption_adjoint_bb} on the adjoint state, the objective values converges q-linearly.
\begin{proposition}\label{prop:optimaldistance_gcg_fast_convergence}
	Suppose that $\ControlOp \in \mathcal{L}(L^1(\omega), H)$. 
	If~\eqref{eq:assumption_adjoint_bb} holds at $\bar{u}(\nu)$,
	then there is $\lambda \in [1/2,1)$ such that
	\begin{equation}
	\label{eq:optimaldistance_gcg_fast_convergence_functional}
	0 \leq f(u_n) - f(\bar{u}) \leq 
	\left[f(u_0) - f(\bar{u})\right]\lambda^n,\quad n \geq 0.
	\end{equation}
	The constant $\lambda$ exclusively depends on $C$, $u_a$, $u_b$, $\omega$, and the Lipschitz constant of $f'$ on $\Qad(0,1)$. 
	Moreover, for a constant $c > 0$ we have
	\begin{equation}
	\label{eq:optimaldistance_gcg_fast_convergence_controls}
	\norm{u_n - \bar{u}}_{L^1(I\times\omega)} \leq c\lambda^{n/2},\quad n \geq 0.
	\end{equation}
\end{proposition}
\begin{proof}
	Since $\ControlOp \colon L^1(\omega) \to H$, the variation of constants formula implies that
	the control-to-state mapping is linear and continuous from $L^1(I\times\omega)$ to $C([0,1]; H)$. Hence, $f$ as a mapping defined on $L^1(I\times\omega)$ is (infinitely often) continuously differentiable. 
	Furthermore, \cref{prop:derivativeF_lowerbound_L1} implies
	\[
	f'(\bar{u})(u-\bar{u}) \geq c_0\nu\norm{u-\bar{u}}^2_{L^1(I\times\omega)}\quad\text{for all~} u \in \Qad(0,1),
	\]
	for some constant $c_0 > 0$. Therefore, \eqref{eq:optimaldistance_gcg_fast_convergence_functional} follows from \cite[Theorem~3.1~(iii)]{Dunn1980}.
	Finally, convexity of $f$ and the inequality above yield~\eqref{eq:optimaldistance_gcg_fast_convergence_controls}.	
\end{proof}

\subsection{Accelerated conditional gradient method for the inner optimization}
\label{subsec:conditional_gradient_acc}

Since the criterion from \cref{prop:optimaldistance_gcg_fast_convergence} 
guaranteeing q-linear convergence of the conditional
gradient method is not satisfied for many examples
and we in fact observe slow convergence in practice,
we employ an acceleration strategy that is described in the following:
Instead of minimizing the convex combination of the last iterate $u_n$ and the new point $u_{n+1/2}$ in \eqref{eq:optimaldistance_gcg_lambda}, we search for the best convex combination of all previous iterates plus the new point $u_{n+1/2}$. 
Concretely, instead of~\eqref{eq:optimaldistance_gcg_lambda}
we determine $\lambda^*$ as
\begin{equation}
\label{eq:optimaldistance_gcg_lambda_acc}
\lambda^* = \argmin_{\lambda \in \ProbSimplex{n+2}} f(\sum_{i=0}^{n}\lambda_i u_i + \lambda_{n+1} u_{n+1/2}),
\end{equation}
where $\ProbSimplex{n+2} = \set{\lambda \in \R^{n+2} \colon \lambda_i \geq 0 
	\text{~and~} \sum_{i=0}^{n+1}\lambda_i = 1}$
denotes the probability simplex in $\R^{n+2}$.
The next iterate is then defined
as $u_{n+1} = \sum_{i=0}^{n}\lambda_i^* u_i + \lambda_{n+1}^* u_{n+1/2}$.
In order to derive an efficient algorithm for the determination of $\lambda^*$, 
we first reformulate the optimality condition associated with~\eqref{eq:optimaldistance_gcg_lambda_acc}
employing the normal map due to Robinson \cite{Robinson1992}.
To this end, let us abbreviate $h(\lambda) \ldef f(\sum_{i=0}^{n}\lambda_i u_i + \lambda_{n+1} u_{n+1/2})$.
For any $c > 0$ we define Robinson's normal map as
\[
G(\eta) = c(\eta - \ProjSimplex(\eta)) + \nabla h(\ProjSimplex(\eta)),
\]
where $\ProjSimplex$ denotes the projection onto $\ProbSimplex{n+2}$.
Due to convexity of $h$, which follows immediately from the convexity of $f$, 
an optimal solution of~\eqref{eq:optimaldistance_gcg_lambda_acc}
can be characterized by means of the normal map as follows; cf.\ \cite[Prop.~3.5]{Pieper2015}.
\begin{proposition}
	\label{prop:cg_acceleration_normalmap}
	$\lambda^* \in \ProbSimplex{n+2}$ is optimal for~\eqref{eq:optimaldistance_gcg_lambda_acc}
	if and only if there exists $\eta \in \R^{n+2}$ such that
	$G(\eta) = 0$ and $\lambda^* = \ProjSimplex(\eta)$.
\end{proposition}
\begin{proof}
	First of all, $\lambda^*$ is optimal for~\eqref{eq:optimaldistance_gcg_lambda_acc}
	if and only if $h'(\lambda^*)(\lambda - \lambda^*) \geq 0$
	for all $\lambda \in \ProbSimplex{n+1}$, because $h$ is convex.
	Let $\lambda^* \in \ProbSimplex{n+2}$ be optimal
	and set $\eta = \lambda^*-c^{-1}\nabla h(\lambda^*)$.
	Then we have
	\[
	\inner{\eta-\lambda^*, \lambda - \lambda^*}_{\R^{n+2}} 
	= -c^{-1} h'(\lambda^*)(\lambda - \lambda^*) \leq 0	
	\]
	for all $\lambda \in \ProbSimplex{n+2}$.
	Hence, $\lambda^* = \ProjSimplex(\eta)$.
	Moreover, by construction $G(\eta) = 0$.
	Conversely, let $\eta$ be given such that $G(\eta) = 0$
	and $\lambda^* = \ProjSimplex(\eta)$. 
	Then 
	\[
	-c\inner{\eta-\lambda^*, \lambda - \lambda^*}_{\R^{n+2}}
	= h'(\lambda^*)(\lambda - \lambda^*) \geq 0	
	\]
	for all $\lambda \in \ProbSimplex{n+2}$. Thus, $\lambda^*$ is optimal.
\end{proof}
In view of \cref{prop:cg_acceleration_normalmap},
to determine $\lambda^*$ defined in~\eqref{eq:optimaldistance_gcg_lambda_acc},
we can equivalently solve the nonlinear equation $G(\eta) = 0$ for $\eta$
and obtain the optimal solution $\lambda^*$ by projecting $\eta$ onto the probability simplex, i.e.\ $\lambda^* = \ProjSimplex(\eta)$.
We propose to solve the equation $G(\eta) = 0$
by means of a semi-smooth Newton method.
Note that the projection $\ProjSimplex$ and its derivative $D\ProjSimplex$
can be efficiently evaluated (with cost $\mathcal{O}(n \log n)$), see \cref{alg:projection_simplex},
where we have extended the algorithm from \cite{Wang2013}
by adding the derivative $D\ProjSimplex$.
The resulting semi-smooth Newton method is summarized in \cref{alg:optimaldistance_inner_gcg_semismooth_newton}.
\begin{algorithm}
	Choose $c > 0$ and $\eta \in \R^{n+2}$\;
	Set $\lambda = \ProjSimplex(\eta)$\;		
	\While{$\norm{c(\eta - \lambda) + \nabla h(\lambda)}_{\R^{n+2}} > \algtol{tol}$}{
		Calculate $\xi = \left(c(\operatorname{Id} - D\ProjSimplex(\eta)) + \nabla^2 h(\lambda)D\ProjSimplex(\eta)\right)^{-1}\left(c(\eta - \lambda) + \nabla h(\lambda)\right)$\;
		Set $\eta = \eta - \xi$ and $\lambda = \ProjSimplex(\eta)$\;
	}
	\caption{Semi-smooth Newton method for solution of \eqref{eq:optimaldistance_gcg_lambda_acc}}
	\label{alg:optimaldistance_inner_gcg_semismooth_newton}
\end{algorithm}
\begin{algorithm}
	\KwIn{$y \in \R^n$}
	Sort $y$
	such that $y_{\pi(1)} \geq y_{\pi(2)} \geq \ldots \geq y_{\pi(n)}$\;
	Find $\rho = \max\set{1 \leq j \leq n \colon y_{\pi(j)} + \frac{1}{j}\left(1 - \sum_{i = 1}^{j}y_{\pi(i)}\right) > 0}$\;
	Define $\lambda = \frac{1}{\rho}\left(1 - \sum_{i = 1}^{\rho}y_{\pi(i)}\right)$\;
	Set $\Gamma = (\gamma_{i,j})_{i,j} \in \R^{n\times n}$ with $\gamma_{i,i} = 1$ if $x_i + \lambda > 0$ and $\gamma_{i,j} = 0$ otherwise for $1 \leq i,j \leq n$\;
	Set $x = (x_i)_i \in \R^n$ with $x_i = \max\set{y_i + \lambda, 0}$ for $1 \leq i \leq n$\;	
	Set $\Lambda = (\lambda_{i,j})_{i,j}  \in \R^{n\times n}$
	with $\lambda_{i,j} = -1/\rho$ if $j = \pi(k)$ for some $1 \leq k \leq \rho$
	and $\lambda_{i,j} = 0$ otherwise	
	for $1 \leq i, j \leq n$\;
	\KwOut{$\ProjSimplex(y) = x$ and $D\ProjSimplex(y) = \Gamma\left(\operatorname{Id} + \Lambda\right)$}
	\caption{Projection $\ProjSimplex$ onto the probability simplex $\ProbSimplex{n}$ and its derivative}
	\label{alg:projection_simplex}
\end{algorithm}

The accelerated conditional gradient method is exactly \cref{alg:optimaldistance_inner_gcg}
except for the last two lines: The parameter $\lambda^*$ is determined using \cref{alg:optimaldistance_inner_gcg_semismooth_newton} (in contrast to the standard conditional gradient method where we could calculate $\lambda^*$ explicitly).
Moreover, in the last line we set $u_{n+1} = \sum_{i=0}^{n}\lambda_i^* u_i + \lambda_{n+1}^* u_{n+1/2}$.
We note that the accelerated version is at least as fast as
the standard conditional gradient method, 
because the feasible set from~\eqref{eq:optimaldistance_gcg_lambda}
is contained in~\eqref{eq:optimaldistance_gcg_lambda_acc}.
In practice we observe that the acceleration strategy significantly
improves the performance.

\section{Numerical examples}
\label{sec:numerical_examples}
\newlength\convergencePlotWidth
\setlength\convergencePlotWidth{4.0cm}
\newlength\convergencePlotHeight
\setlength\convergencePlotHeight{3.8cm}

As a proof of concept, we implement numerical examples
illustrating that the proposed algorithm can be realized in practice.
We begin with one example governed by an ordinary differential equation,
even though our main focus are systems 
subject to partial differential equations.

Since the value function $\delta(\cdot)$ can be non-convex, 
we consider the damped Newton method.
If $\delta(\nu_{n+1}) < -\algtol{tol}$, then the
Newton step is iteratively multiplied by the damping factor $\gamma = 0.9$
until $\delta(\nu_{n+1}) > -\algtol{tol}$.
Note that this strategy does not require 
the inner problem to be solved with high accuracy.
If a feasible control with sufficiently negative value for $\delta(\cdot)$ is known, 
then the conditional gradient method can be 
restarted with a smaller Newton step.	

Moreover, we have implemented the acceleration strategy 
from \cref{subsec:conditional_gradient_acc}
for the conditional gradient method.
To keep the memory requirements moderate, points that are associated 
with small coefficients in the convex combination
are being removed from the list of former iterates. 
In our examples, this strategy significantly improves the convergence.

\subsection{Linearized pendulum}
\label{subsec:example17_2}
We first consider a time-optimal control example subject to 
an ordinary differential equation from \cite[Example~17.2]{Hermes1969}. 
The operators $A$ and $B$ are given by the matrices
\[
A = \left(\begin{array}{cc}
0 & -1\\
1 & 0
\end{array}\right),
\quad
B = \left(\begin{array}{c}
0\\
1
\end{array}\right).
\]
Hence, we set $V = H = V^* = \R^2$ and $Q = \R^1$.
Moreover, the control constraints are $u_a = -1$ and $u_b = 1$,
and the desired state is $y_d = 0$.
The corresponding state equation describes a harmonic oscillator, 
precisely the linearized pendulum $\ddot{x} + x = u$ with forcing term $u$.
Note that the system is normal, so \eqref{eq:optimaldistance_opt_control} possesses a unique minimizer; see \cref{prop:optimaldistance_unique_continuation,remark:unique_continuation}.
As shown in \cite[Example~17.2]{Hermes1969}, the optimal trajectories for $\delta_0 = 0$
can be constructed geometrically. 
For example, if
\[
y_0 = -r\left(\cos(\uppi/3 - \theta_0), \sin(\uppi/3 - \theta_0)\right)^T + (1,-3)^T, \quad 
\theta_0 = \arcsin(1/r),\quad r = \sqrt{17},
\]
then the optimal trajectory consists of three semi circles 
with $\theta = \uppi/3$ and center $(1,0)^T$, $\theta = \uppi$ and center $(-1,0)^T$, 
and $\theta = \uppi/2$ and center $(1,0)^T$. 
In addition, the optimal time is $T = \uppi (1/3 + 1 + 1/2) = 11\uppi/6$,
and the unique optimal control is given by
\[
\bar{u}(t) = \begin{cases}
1 &\text{if~} 0 \leq t \leq \uppi/3,\\
-1 & \text{if~} \uppi/3 < t \leq 4\uppi/3,\\
1 & \text{if~} 4\uppi/3 < t \leq  11\uppi/6.\\
\end{cases}
\]
The ordinary differential equation is discretized
by means of the discontinuous Galerkin method 
with piecewise constant functions
(corresponding to the implicit Euler method)
for an equidistant time grid with $M$ denoting the 
number of time intervals.	
To solve the problem with our approach, 
we consider a relaxation of the terminal constraint
by taking $\delta_0 = 10^{-6}$.
Since the solution is stable with respect to
perturbations in the constraint, 
the relaxation has no significant influence on the optimal solution,
as long as the error due to the discretization of
the state equation dominates the overall error.

As depicted in \cref{fig:example17_2_valuefct}
we observe fast convergence of the Newton method.
Moreover, the number of Newton steps in the outer loop
and the number of iterations of the conditional gradient
method in the inner loop seem to be essentially independent of
the discretization of the state equation; see \cref{table:example1_convergence}.

\begin{figure}[!ht]
	\begin{center}		
		\includegraphics{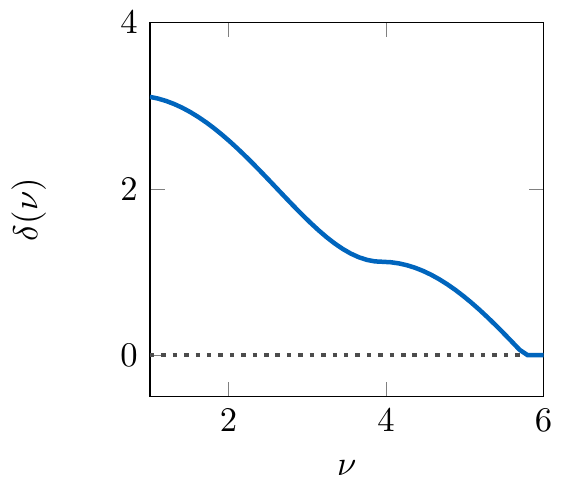}
		\includegraphics{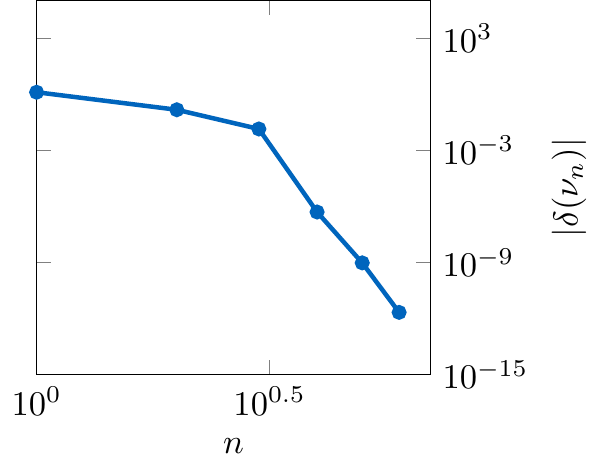}
	\end{center}
	\caption{Value function $\delta$ (left) and absolute value of $\delta(\nu_n)$ (right)
		for $\nu_n$ the iterates generated by the Newton method (\cref{alg:optimaldistance_outer_newton}) for Example~\ref{subsec:example17_2}
		with $M = 10\,000$ time steps for the implicit Euler method.}
	\label{fig:example17_2_valuefct}
\end{figure}

\begin{table}
	\begin{center}
		\begin{tabular}{rrrrr}
			\toprule
			$M$ & $T_k$ & $\abs{T_k-T}$ & Newton steps & cG steps\\
			\midrule
			$100$ & $5.501553$ & $2.5803_{-1}$ & $6$ ($1$) & $44$\\ 
			$1\,000$ & $5.730029$ & $2.9557_{-2}$ & $6$ ($1$) & $52$\\ 
			$10\,000$ & $5.756636$ & $2.9504_{-3}$ & $6$ ($1$) & $58$\\ 
			$100\,000$ & $5.759346$ & $2.4051_{-4}$ & $6$ ($1$) & $56$\\ 
			$1\,000\,000$ & $5.759618$ & $3.1244_{-5}$ & $6$ ($1$) & $53$\\ 
			\bottomrule
		\end{tabular}
	\end{center}
	\caption{Computed optimal times, absolute errors, number of Newton steps in outer loop 
		(number of damped steps in brackets), 
		and number of conditional gradient steps in inner loop
		for Example~\ref{subsec:example17_2}
		with $M$ denoting the number of time steps.
		Moreover, $\delta_0 = 10^{-6}$ and the inital value for the Newton method is $\nu_0 = 0.6\,T$.}
	\label{table:example1_convergence}
\end{table}

\subsection{Linear heat-equation with distributed control}
\label{subsec:example_HeatHat2}
Next, we consider the following problem subject to
the linear heat-equation. Let
\begin{align*}
\Omega &= (0,1)^2\mbox{,}\quad
\omega = (0.25, 0.75)^2\mbox{,}\quad \delta_0 = 1/10,\\
y_0(x) &= 4\sin(\pi x_1^2) \sin(\pi x_2)^3,\quad y_d(x) = -2\min\set{x_1, 1-x_1,x_2,1-x_2}\mbox{,}\\[0.5em]
\Qad(0,1) &= \set{u \in L^2(I\times\omega) \constraintSet -5 \leq u \leq 0}\mbox{.}
\end{align*}
Moreover, $A = -0.03\Lap$ with $-\Lap$ the Laplace operator equipped with 
homogeneous Dirichlet boundary conditions. 
The control operator $B$ is the extension by zero operator.	
Hence, we take $V = H^1_0(\Omega)$, $H = L^2(\Omega)$, $V^* = H^{-1}(\Omega)$,
and $\Q = L^2(\omega)$.
Note that the control acts on a subset $\omega \subset \Omega$, only. 
Concerning the practical implementation,
we consider a discontinuous Galerkin method
in time and a continuous Galerkin method in space.
The state and adjoint state equations are discretized by means of 
piecewise constant functions in time (corresponding to
the implicit Euler method) and continuous and cellwise linear 
functions in space.

\begin{table}
	\begin{center}
		\begin{tabular}{r rrrrr}
			\toprule
			$M$ & $N$ & $T_k$ & $\abs{T_k-T}$ & Newton steps & cG steps\\
			\midrule
			$20$ & $4225$ & $1.573876$ & $9.0814_{-2}$ & $6$ ($0$) & $206$\\ 
			$40$ & $4225$ & $1.525617$ & $4.2554_{-2}$ & $6$ ($0$) & $210$\\ 
			$80$ & $4225$ & $1.501899$ & $1.8836_{-2}$ & $6$ ($0$) & $209$\\ 
			$160$ & $4225$ & $1.490127$ & $7.0645_{-3}$ & $6$ ($0$) & $210$\\ 
			$320$ & $4225$ & $1.484246$ & $1.1830_{-3}$ & $6$ ($0$) & $196$\\
			\midrule 
			$640$ & $81$ & $1.291744$ & $1.9132_{-1}$ & $6$ ($0$) & $133$\\ 
			$640$ & $289$ & $1.423571$ & $5.9492_{-2}$ & $6$ ($0$) & $162$\\ 
			$640$ & $1089$ & $1.468758$ & $1.4304_{-2}$ & $6$ ($0$) & $188$\\ 
			$640$ & $4225$ & $1.481302$ & $1.7608_{-3}$ & $6$ ($0$) & $190$\\ 
			\bottomrule
		\end{tabular}
	\end{center}
	\caption{Computed optimal times, absolute errors, number of Newton steps in outer loop 
		(number of damped steps in brackets), 
		and number of conditional gradient steps in inner loop
		for Example~\ref{subsec:example_HeatHat2}
		with $M$ denoting the number of time steps
		and $N$ the number of nodes for the spatial discretization.
		Moreover, the inital value for the Newton method is $\nu_0 = 0.8$.}
	\label{table:example3_convergence}
\end{table}

\begin{figure}[!ht]
	\begin{center}
		\includegraphics{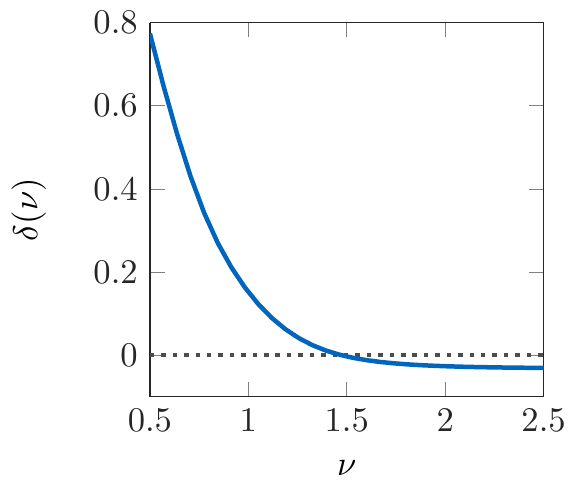}
		\includegraphics{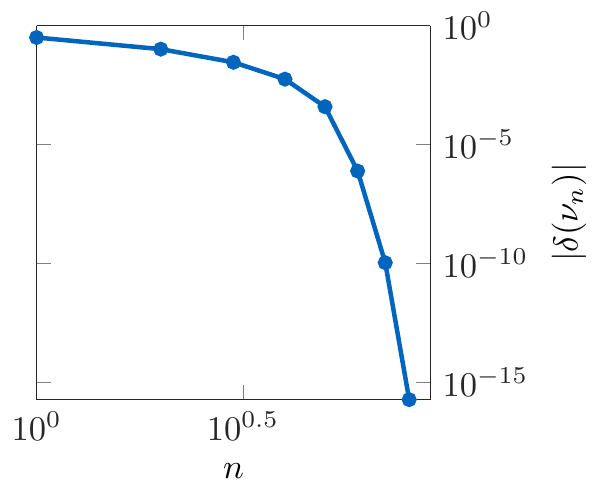}
	\end{center}
	\caption{Value function $\delta$ (left) and absolute value of $\delta(\nu_n)$ (right)
		for $\nu_n$ the iterates generated by the Newton method (\cref{alg:optimaldistance_outer_newton}) for Example~\ref{subsec:example_HeatHat2}.}
	\label{fig:example_HeatHat2}
\end{figure}

As in the first example,
we observe fast convergence of the Newton method, see \cref{table:example3_convergence,fig:example_HeatHat2}.
Moreover, we observe quadratic order of convergence
with respect to the spatial discretization
and linear order convergence
with respect to the temporal discretization.
For further details and a priori discretization error estimates
we also refer to \cite{Bonifacius2018b}.

Before turning to the next example, 
we would like to compare the algorithm from \cref{sec:algorithm} 
to an alternative approach, where the time-optimal control problem~\eqref{Pt}
is solved directly after adding a regularization term  to the objective functional,
precisely the $L^2$-norm of the control
variable.
Clearly, we are interested in steering 
the regularization parameter to zero.
The terminal constraint in~\eqref{Pt} is treated algorithmically
by means of the augmented Lagrange method. 
The resulting optimization problems are solved
by means of a semi-smooth Newton method 
in a monolithic way, i.e.\ we consider the tuple $(\nu, u)$ as
a joint optimization variable; cf.\ \cite{Kunisch2015} and \cite[Section~4.1]{Bonifacius2018}.
We observe that our approach requires roughly four to ten times less
solves of the PDE than the regularization approach for any fixed 
regularization parameter in the range from $\alpha = 0.001$
to $\alpha = 10$; see \cref{table:example2_costs}.
Employing a path-following strategy, 
where one iteratively decreases the regularization parameter
starting with a moderate value of $\alpha$
and uses the solution of the former iteration
as the initial value for the next optimization 
(see, e.g., \cite{Hintermueller2006a}),
one could avoid the high computational costs for small $\alpha$.
However, this strategy requires at least one solution 
without warm start, so that our approach
is (in this example) at least five times faster.
\begin{table}
	\begin{center}
		\begin{tabular}{rrcccccc}
			\toprule
			& & Min.\ dist.\ & \multicolumn{5}{c}{Augmented Lagrange method with regularization}\\
			$M$ & $N$ &  & $\alpha = 0.001$ & $\alpha = 0.01$ & $\alpha = 0.1$ & $\alpha = 1$ & $\alpha = 10$\\
			\midrule
			$20$ & $4225$ & $260$ & $2681$ & $1598$ & $1208$ & $1375$ & $1813$\\ 
			$40$ & $4225$ & $250$ & $2703$ & $1447$ & $1177$ & $1440$ & $1675$\\ 
			$80$ & $4225$ & $250$ & $2687$ & $1424$ & $1130$ & $1434$ & $1685$\\ 
			$160$ & $4225$ & $252$ & $3017$ & $1697$ & $1154$ & $1388$ & $1705$\\ 
			$320$ & $4225$ & $222$ & $3068$ & $1587$ & $1100$ & $1356$ & $1709$\\ 
			\midrule
			$640$ & $81$ & $214$ & $1879$ & $1091$ & $924$ & $1074$ & $1215$\\ 
			$640$ & $289$ & $198$ & $1824$ & $1070$ & $774$ & $1180$ & $1479$\\ 
			$640$ & $1089$ & $212$ & $2495$ & $1412$ & $1028$ & $1372$ & $1789$\\ 
			$640$ & $4225$ & $246$ & $3597$ & $1647$ & $1050$ & $1364$ & $1781$\\ 
			\bottomrule
		\end{tabular}
	\end{center}
	\caption{Number of PDE solves for the algorithm from \cref{sec:algorithm} 
		based on solving minimal distance problems
		and the augmented Lagrange method with $L^2$-regularization for the control
		and $\alpha$ the regularization parameter.
		$M$ denotes the number of intervals for the temporal discretization
		and $N$ the number of nodes for the spatial discretization of the PDE.
		The initial parameters for the augmented Lagrange method are $c_0 = 2\cdot 10^4$
		and	$\mu_0 = 80$.}
	\label{table:example2_costs}
\end{table}

\subsection{Linear heat-equation with Neumann boundary control}
\label{subsec:example_HeatNeumann1}

\newlength\neumannBoundaryControlPlotWidth
\setlength\neumannBoundaryControlPlotWidth{13.0cm}
\newlength\neumannBoundaryControlPlotHeight
\setlength\neumannBoundaryControlPlotHeight{4.0cm}

Last, we consider the following problem subject to
the linear heat-equation with Neumann boundary control. 
Concretely, let
\begin{align*}
\Omega &= (0,1)^2\mbox{,}\quad
\omega = \partial\Omega\mbox{,}\quad \delta_0 = 1/10,\\
y_0(x) &= 4\sin(\pi x_1^2) \sin(\pi x_2)^3,\quad y_d(x) = 0\mbox{,}\\[0.5em]
\Qad(0,1) &= \set{u \in L^2(I\times\omega) \constraintSet -5 \leq u \leq 5}\mbox{.}
\end{align*}
Moreover, $A = -0.03\Lap$ with $-\Lap$ the Laplace operator.
The control operator $B$ is the adjoint of the trace operator,
i.e.\ $B = \trace^* \colon L^2(\partial\Omega) \to (H^{1}(\Omega))^*$.	
Hence, we take $V = H^1(\Omega)$, $H = L^2(\Omega)$, $V^* = (H^{1}(\Omega))^*$,
and $\Q = L^2(\omega)$.
We consider the same discretization scheme
for the state and adjoint state equation as before.
Moreover, the control is discretized by edge-wise constant functions on the boundary.

\begin{figure}[!ht]
	\begin{center}
		\includegraphics{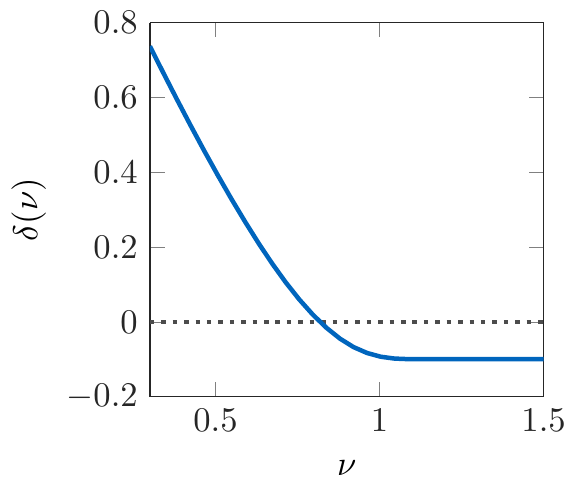}
		\includegraphics{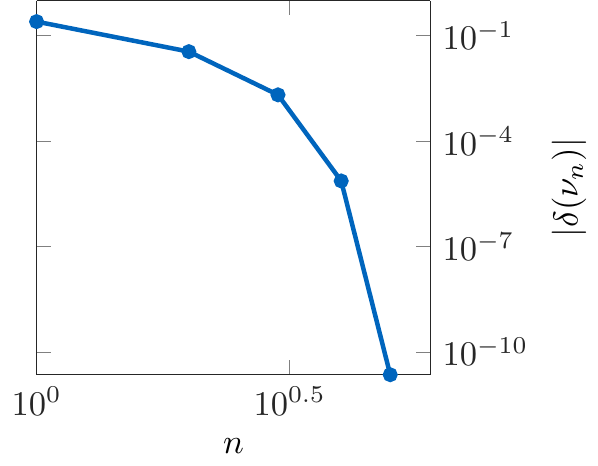}
	\end{center}
	\caption{Value function $\delta$ (left) and absolute value of $\delta(\nu_n)$ (right)
		for $\nu_n$ the iterates generated by the Newton method (\cref{alg:optimaldistance_outer_newton}) for Example~\ref{subsec:example_HeatNeumann1}.}
	\label{fig:example_HeatNeumann1}
\end{figure}

\begin{table}
	\begin{center}
		\begin{tabular}{r rrrrr}
			\toprule
			$M$ & $N$ & $T_k$ & $\abs{T_k-T}$ & Newton steps & cG steps\\
			\midrule
			$20$ & $4225$ & $8.8971_{-1}$ & $4.0493_{-2}$ & $4$ ($0$) & $3302$\\ 
			$40$ & $4225$ & $8.6784_{-1}$ & $1.8619_{-2}$ & $4$ ($0$) & $4141$\\ 
			$80$ & $4225$ & $8.5681_{-1}$ & $7.5882_{-3}$ & $4$ ($0$) & $4979$\\ 
			$160$ & $4225$ & $8.5134_{-1}$ & $2.1231_{-3}$ & $4$ ($0$) & $6294$\\
			\midrule 
			$320$ & $81$ & $6.9746_{-1}$ & $1.5176_{-1}$ & $3$ ($0$) & $1938$\\ 
			$320$ & $289$ & $8.1011_{-1}$ & $3.9113_{-2}$ & $4$ ($0$) & $5265$\\ 
			$320$ & $1089$ & $8.4096_{-1}$ & $8.2584_{-3}$ & $4$ ($0$) & $8029$\\ 
			$320$ & $4225$ & $8.4865_{-1}$ & $5.7155_{-4}$ & $4$ ($0$) & $8432$\\ 
			\bottomrule
		\end{tabular}
	\end{center}
	\caption{Computed optimal times, absolute errors, number of Newton steps in outer loop 
		(number of damped steps in brackets), 
		and number of conditional gradient steps in inner loop
		for Example~\ref{subsec:example_HeatNeumann1}
		with $M$ denoting the number of time steps
		and $N$ the number of nodes for the spatial discretization.
		Moreover, the initial value for the Newton method is $\nu_0 = 0.6$.}
	\label{table:example4_convergence}
\end{table}

The optimal control obtained numerically is depicted in \cref{fig:example_HeatNeumann1_control},
where the boundary of the square domain has been unrolled.
Note that switching hyperplanes of the control seem to accumulate 
towards the end of the time horizon.
As in the preceding examples, we observe fast convergence of 
the Newton method for the outer loop; see \cref{fig:example_HeatNeumann1,table:example4_convergence}.

\begin{figure}[!ht]
	\begin{center}
		\includegraphics{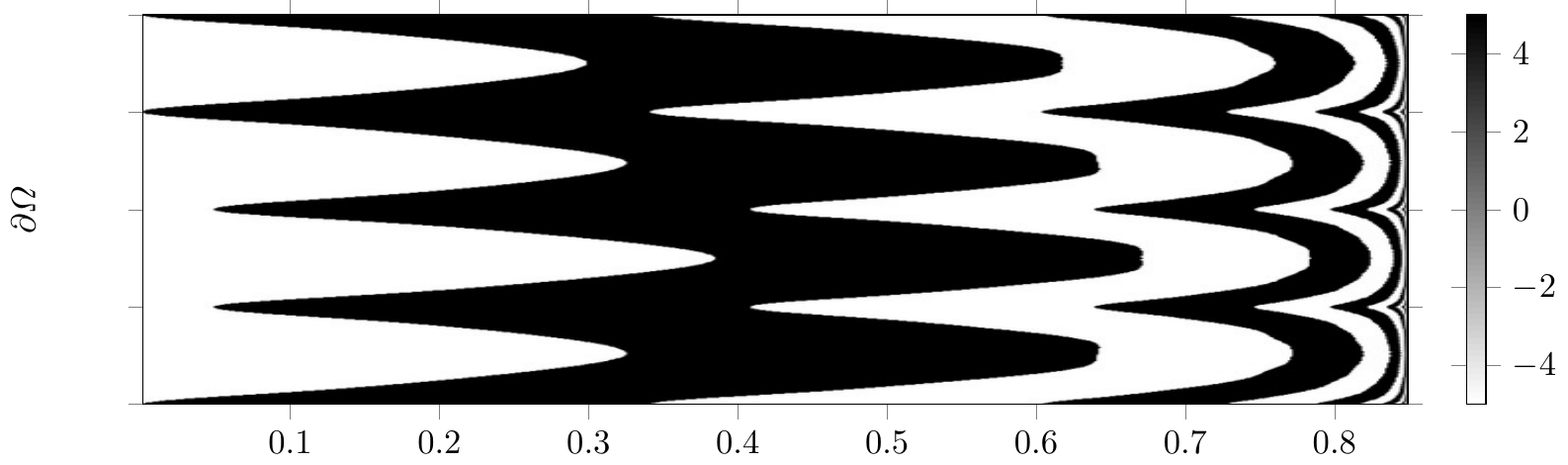}
	\end{center}
	\caption{Optimal control for Example~\ref{subsec:example_HeatNeumann1}. 
		Black denotes the upper bound and white the lower bound of the control constraints.}
	\label{fig:example_HeatNeumann1_control}
\end{figure}

\section{Open problems}
\label{sec:open_problems}
We conclude with some open problems.

\begin{enumerate}[(i)]
	\item
	To prove the equivalence of time-optimal
	and distance-optimal controls, we required that $T(\cdot)$
	is left-continuous; see \cref{thm:equivalence_topt_dopt}.
	We stated two sufficient conditions;
	see \cref{prop:T_continuousfromleft_sufficient,prop:T_Lipschitzfromleft_sufficient}.
	The latter can be checked a priori without knowing an optimal solution,
	whereas the first depends on a certain controllability condition
	under pointwise control constraints
	that is difficult to verify.	
	It would be desirable to know further sufficient conditions
	that can be easily verified for concrete problems.
	
	\item
	Moreover, to strengthen the directional derivative of $\delta(\cdot)$
	to a classical derivative, 
	one has to ensure that the integral expression
	in~\eqref{eq:optimaldistance_value_function_derivative} is 
	independent of the control variable.
	This is guaranteed for purely time-dependent controls (see \cref{cor:optimaldistance_value_function_parameter_control})
	or if a backwards uniqueness property holds (see \cref{cor:optimaldistance_value_function_backwards_uniqueness}).
	Clearly, the backwards uniqueness property 
	of other control scenarios
	is of independent interest
	and would also lead to more applications for our approach.
	
	\item
	Last, Lipschitz continuity of $\delta'(\cdot)$
	yields fast local convergence of the Newton method,
	which further justifies to use the equivalence
	of time-optimal and distance-optimal controls
	for numerical realization. 
	Here, we only stated one sufficient condition
	that relies on the structural assumption of 
	the adjoint state~\eqref{eq:assumption_adjoint_bb}; see \cref{prop:newton_method_convergence_quadratic}. 
	This condition does not seem to be sufficient as
	we observe Lipschitz continuity of $\delta'(\cdot)$ in the numerical examples
	even if~\eqref{eq:assumption_adjoint_bb} is violated.
	Different techniques to show Lipschitz continuity of $\delta'(\cdot)$
	would require a second order sufficient optimality condition.
	However, such a condition cannot be expected to hold
	in the case of bang-bang controls.
\end{enumerate}

\addcontentsline{toc}{section}{Bibliography}
\bibliographystyle{abbrv}

\end{document}